\documentclass[a4paper,11pt]{article}

\usepackage{xcolor}
\usepackage[pdfencoding=unicode, hidelinks]{hyperref}
\hypersetup{
	colorlinks,
	linkcolor={red},
	citecolor={blue},
}
\usepackage[english]{babel}
\usepackage[utf8]{inputenc}
\usepackage{amsthm}
\usepackage{amsmath}
\usepackage{amssymb}
\usepackage{mathrsfs}
\usepackage[mathscr]{eucal}
\usepackage{wasysym}
\usepackage{textcomp}
\usepackage{enumerate}
\usepackage{setspace}
\usepackage{latexsym}
\usepackage{graphicx}
\usepackage{mathtools}
\usepackage{tabularray}
\usepackage{geometry}
\usepackage{float}
\usepackage{emptypage}
\usepackage{commath}
\usepackage{multicol}
\usepackage{titlesec}
\usepackage[shortlabels]{enumitem}
\usepackage{verbatim}
\usepackage{cases}
\usepackage{cleveref}
\usepackage{orcidlink}
\usepackage{authblk}

\usepackage{csquotes}
\usepackage[backend=biber,
    style=alphabetic,
    citestyle=alphabetic,
    doi=false,
    url=false,
    isbn=false,
    maxnames=6,
    abbreviate=true,
    giveninits=true]{biblatex}
\usepackage{hyperref}
\AtEveryBibitem{%
  \clearfield{note}%
}

\theoremstyle{plain}
\newtheorem{thm}{Theorem}[section]
\theoremstyle{plain}
\newtheorem{cor}[thm]{Corollary}
\theoremstyle{plain}
\newtheorem{lemma}[thm]{Lemma}
\theoremstyle{plain}
\newtheorem{defn}[thm]{Definition}
\theoremstyle{plain}
\newtheorem{prop}[thm]{Proposition}
\theoremstyle{definition}
\newtheorem{remark}[thm]{Remark}
\newtheorem{ex}[thm]{Example}

\numberwithin{equation}{section}

\newcommand{\RR}{\mathbb{R}}

\newcommand{\NN}{\mathbb{N}}
\newcommand{\dx}{\,\ensuremath{\mathrm{d}}x\,}
\newcommand{\ds}{\,\ensuremath{\mathrm{d}}s\,}
\newcommand{\dt}{\,\ensuremath{\mathrm{d}}t\,}
\newcommand{\dxt}{\,\ensuremath{\mathrm{d}}x \, \ensuremath{\mathrm{d}}t}
\newcommand{\dtau}{\,\ensuremath{\mathrm{d}}\tau\,}
\newcommand{\dA}{\,\ensuremath{\mathrm{d}}\mathcal{H}^{d-1}\,}
\renewcommand{\norm}[1]{\lVert#1\rVert}
\newcommand{\duality}[2]{\langle#1,#2\rangle}

\newcommand{\sigmag}{\sigma_{\Gamma}}

\newcommand{\A}{\mathcal{A}}
\newcommand{\B}{\mathcal{B}}
\newcommand{\uu}{\boldsymbol{u}}

\newcommand{\epsilonu}{\varepsilon(\boldsymbol{u})}
\newcommand{\epsilonut}{\varepsilon(\partial_t\uu)}

\newcommand{\oomega}{\boldsymbol{\omega}}
\newcommand{\oomegat}{\partial_t \oomega}
\newcommand{\epsilono}{\varepsilon(\oomega)}
\newcommand{\epsilonot}{\varepsilon(\oomegat)}
\newcommand{\ttheta}{\boldsymbol{\theta}}
\newcommand{\cchi}{\boldsymbol{\chi}}

\newcommand{\ff}{\boldsymbol{f}}
\newcommand{\nnu}{\boldsymbol{\nu}}

\newcommand{\Ws}{W}
\newcommand{\Wzs}{W_0} 
\newcommand{\Vs}{V}
\newcommand{\Hs}{H}

\DeclareMathOperator{\diver}{div}
\newcommand{\into}{\int_{\Omega}}
\newcommand{\intTo}{\int_0^T \!\!\!\into}
\newcommand{\intT}{\int_0^T}

\newcommand{\intt}{\int_0^t}

\newcommand{\ints}{\int_0^s}
\newcommand{\intg}{\int_{\Gamma}}

\newcommand{\U}{\mathcal{U}}
\newcommand{\Uad}{\U_{\text{ad}}}
\newcommand{\Ur}{\U_{R}}
\renewcommand{\S}{\mathcal{S}}
\newcommand{\V}{\mathcal{V}}
\newcommand{\hh}{\boldsymbol{h}}
\newcommand{\yy}{\boldsymbol{y}}
\newcommand{\C}{\mathfrak{C}}
\newcommand{\cc}{\mathfrak{c}}
\newcommand{\dd}{\mathfrak{d}}
\newcommand{\DD}{\mathfrak{D}}
\newcommand{\varphih}{\varphi^{\hh}}
\newcommand{\sigmah}{\sigma^{\hh}}
\newcommand{\uuh}{\uu^{\hh}}
\newcommand{\yyh}{\yy^{\hh}}
\newcommand{\zh}{z^{\hh}}
\newcommand{\xih}{\xi^{\hh}}
\newcommand{\rhoh}{\rho^{\hh}}
\newcommand{\oomegah}{\oomega^{\hh}}
\newcommand{\zetah}{\zeta^{\hh}}
\newcommand{\Phih}{\Phi^{\hh}}
\newcommand{\lambdah}{\lambda^{\hh}}
\newcommand{\wwh}{\boldsymbol{w}^{\hh}}
\newcommand{\muh}{\mu^{\hh}}
\newcommand{\epsilonuh}{\varepsilon(\uu^{\hh})}

\newcommand{\epsilonoh}{\varepsilon(\oomega^{\hh})}

\newcommand{\epsilonwh}{\varepsilon(\wwh)}
\newcommand{\epsilonwth}{\varepsilon(\partial_t\wwh)}
\newcommand{\q}{\mathrm{q}}
\renewcommand{\r}{\mathrm{r}}
\newcommand{\vv}{\mathrm{\mathbf{v}}}
\newcommand{\epsilonv}{\varepsilon(\vv)}
\newcommand{\vvt}{\partial_t\vv}
\newcommand{\epsilonvt}{\varepsilon(\vvt)}
\newcommand{\s}{\mathrm{s}}
\newcommand{\J}{\mathrm{J}}
\newcommand{\VS}{\V_{\S}}

\crefname{lemma}{lemma}{lemmas}
\Crefname{lemma}{Lemma}{Lemmas}
\crefname{prop}{proposition}{proposition}
\Crefname{prop}{Proposition}{Propositions}
\crefname{cor}{corollary}{corollaries}
\Crefname{cor}{Corollary}{Corollaries}
\crefname{remark}{remark}{remarks}
\Crefname{remark}{Remark}{Remarks}
\crefname{thm}{theorem}{theorems}
\Crefname{thm}{Theorem}{Theorems}
\Crefname{section}{Section}{Sections}

\usepackage[symbol]{footmisc}
\title{Optimal control on a brain tumor growth model with lactate metabolism, viscoelastic effects, and tissue damage}

\author{\large{\textsc{Giulia Cavalleri}} \orcidlink{0009-0006-4154-9659}}

\affil{\normalsize{Dipartimento di matematica “F. Casorati”, Università degli Studi di Pavia,\\
via Ferrata 5, 27100 Pavia, Italy\\
E-mail: \texttt{giulia.cavalleri01@universitadipavia.it}}}

\author{\large{\textsc{Alain Miranville}} \orcidlink{0000-0002-6030-5928}}

\affil{\normalsize{School of Mathematics and Statistics, Henan Normal University, Xinxiang, P. R. China \& Université Le Havre Normandie, Laboratoire de Mathématiques Appliquées du Havre (LMAH), 25, rue Philippe Lebon, BP 1123, 76063 Le Havre cedex, France\\
E-mail: \texttt{alain.miranville@univ-lehavre.fr}}}

\date{}

\begin{document}

\maketitle

\begin{abstract}
    \noindent In this paper, we study an optimal control problem for a brain tumor growth model that incorporates lactate metabolism, viscoelastic effects, and tissue damage. The PDE system, introduced in \cite{Cavalleri_Colli_Miranville_Rocca_2025}, couples a Fisher–Kolmogorov-type equation for tumor cell density with a reaction-diffusion equation for the lactate, a quasi-static force balance governing the displacement, and a nonlinear differential inclusion for tissue damage. The control variables, representing chemotherapy and a lactate-targeting drug, influence tumor progression and treatment response.  Starting from well-posedness, regularity, and continuous dependence results proven in \cite{Cavalleri_Colli_Miranville_Rocca_2025}, we define a suitable cost functional and prove the existence of optimal controls. Then, we analyze the differentiability of the control-to-state operator and establish a necessary first-order condition for treatment optimality.
            
    \vskip3mm
    

    \vskip3mm
    
		
\end{abstract}

\section{Introduction}

Primary brain tumors, especially gliomas and glioblastomas, are characterized by angiogenesis, invasive growth,
and necrosis. Despite advances in medical research and treatment strategies, the brain tumor survival rate remains low (see, e.g., \cite{Delgado_2016}), and no significant improvement has been observed recently (see, e.g., \cite{Szopa_2017}). 
For this reason, it is crucial to explore new therapeutic approaches and to complement clinical studies with mathematical modeling. The aim of the latter is not only to predict the course of the disease, but also to optimize treatments, by adjusting the combination, timing, and dosage of therapies to achieve the best possible outcome for the patient, meaning maximum reduction of the tumor and minimum drug-related side effects. This is where the present contribution comes into play. Specifically, we are interested in a protocol that combines chemotherapy with a drug targeting lactate, which seems promising from an experimental point of view (cf.~\cite{Sonveaux_etal_08}, \cite{Tang_etal_2021}, \cite{Guyon_etal_2022}). 
From the mathematical perspective, only a few works go in this direction (see, e.g., \cite{Aubert_05}, \cite{Cloutier2009}, \cite{Guillevin_etal_2018}, \cite{Cherfils_etal_2021}, \cite{Cherfils_etal_2022}) and even fewer regard optimal control (cf.~\cite{Cherfils_etal_2024}).
This paper seeks to study an optimal control problem starting from the model recently proposed and analyzed in \cite{Cavalleri_Colli_Miranville_Rocca_2025}, which describes the dynamics of a brain tumor taking into account the lactate metabolism, viscoelastic effects of the tissue, and possible damage caused by surgery.  The introduction of the damage variable in this context was inspired by \cite{cavalleri_2024}, by one of the authors of the present contribution, where damage was first incorporated into a tumor model. Among the huge literature regarding optimal control for tumor growth models (see, e.g., 
\cite{Benosman_2015}, \cite{Colli_etal_2021}, 
\cite{Garcke_Lam_2016}, \cite{Garcke_etal_2018},  \cite{Ebenbeck_Knopf_2019}, \cite{Signori_2020}, 
and the references cited therein) we would like to recall once again \cite{Cherfils_etal_2024}, where the authors deal with a brain tumor-specific model. The resulting PDE system couples three parabolic equations, one for the tumor cell density, and the other two for the intracellular lactate concentration and the capillary lactate concentration, respectively. Another perspective related to the present work is the one from \cite{Garcke_Lam_Signori_Opt_Contr_2021}, in which the authors consider a phase-field model of Cahn--Hilliard type taking into account the presence of a nutrient (such as oxygen or glucose) and mechanical effects. \\

\noindent \textbf{The model.} Let $\Omega$ be a bounded $C^2$ domain in $\RR^n$ with $n=2,3$ and let $\nnu$ be its outward unit normal to $\partial \Omega$. We are interested in the PDE system
\begin{subequations}\label{eq:problem}
    \begin{align}
        & \partial_t\varphi - \Delta \varphi = U(\varphi,\sigma,z,\chi_1), \label{eq:phi}\\
        & \partial_t\sigma - \Delta \sigma +K(\varphi,\sigma,z)= \chi_2 \color{black}S(\varphi,z), \label{eq:lactate}\\
        &  - \diver\left[\A \epsilonut + \B(\varphi,z)\epsilonu  \right] = \ff, \label{eq:displacement}\\
        & \partial_t z - \Delta z + \beta(z) + \pi(z)  \ni  \iota - \Psi(\varphi,\epsilonu), \label{eq:damage}
    \end{align}
\end{subequations}
posed in the parabolic cylinder $Q \coloneqq \Omega \times (0,T)$, where the domain $\Omega$ represents the brain and the fixed time $T>0$ is the medical treatment duration. In equations \eqref{eq:phi} and \eqref{eq:lactate} we introduced the notation
\begin{equation*}
    K(\varphi,\sigma,z)\coloneqq \frac{k_1(\varphi,z)\sigma}{k_2(\varphi,z) + \sigma}, \quad U(\varphi,\sigma,z,\chi_1)\coloneqq(p(\sigma,z)- \chi_1 \color{black} ) \varphi\left(1-\frac{\varphi}{N}\right) - \varphi g(\sigma,z) 
\end{equation*}
for the sake of brevity. As already pointed out, a slightly modified version of \eqref{eq:problem} was proposed in \cite{Cavalleri_Colli_Miranville_Rocca_2025}. We refer the reader to it for further modeling details and recall only the most important aspects. The variable $\varphi$ represents the tumor cell density and takes value in $[0,N]$, where the constant $N$ is the fixed carrying capacity of the tissue. It is ruled by the Fischer--Kolmogorov type equation \eqref{eq:phi}. The parabolic reaction-diffusion equation \eqref{eq:lactate} governs the evolution of $\sigma$, which is the intracellular lactate concentration. The quasi-static balance of forces \eqref{eq:displacement} for a viscoelastic tissue describes the dynamic of $\uu$, which is the small displacement field of each point with respect to a reference undeformed configuration. The local tissue damage $z$ takes value in the interval $[0,1]$, where $z=0$ means there is no damage while $z=1$ means that the tissue is totally damaged. Its evolution is ruled by the parabolic differential inclusion \eqref{eq:damage}. Finally, $\chi_1$ is the concentration of a cytotoxic drug, which decreases the tumor proliferation $p(\sigma,z)$ in the mass source $U$. Similarly, the term $\chi_2$ represents a lactate targeting drug and affects the lactate source term $S$. The system \eqref{eq:problem} is coupled with the boundary conditions
\begin{subequations}\label{eq:boundary_conditions}
    \begin{align}
        & \partial_{\nnu}\varphi = \partial_{\nnu}z = 0,\label{eq:phi,z_neumann}\\
        & \partial_{\nnu}\sigma =\sigma_{\Gamma} -\sigma,\label{eq:sigma_robin}\\
        & \uu = \mathbf{0}, \label{eq:u_dirichlet}
    \end{align}
\end{subequations}
and the initial conditions
\begin{equation}\label{eq:initial_conditions}
    \varphi(0) = \varphi_0,\quad
    \sigma(0) = \sigma_0, \quad
    \uu(0) = \uu_0, \quad
    z(0) = z_0.
\end{equation}

\noindent \textbf{The cost functional.} Introducing the notation $\cchi = (\chi_1,\chi_2)$, we define the following cost functional
\begin{equation} \label{eq:cost_functional}
    \begin{split}
        \mathcal{J}((\varphi,\sigma,\uu,z),\cchi) \coloneqq &\frac{\alpha_1}{2} \norm{\varphi-\varphi_Q}_{L^2(Q)}^2 + \frac{\alpha_2}{2} \norm{\varphi(T)-\varphi_{\Omega}}_{\Hs}^2 + \alpha_3 \into \varphi(T) \dx \\
        & + \frac{\alpha_4}{2} \norm{\sigma-\sigma_Q}_{L^2(Q)}^2 + \frac{\alpha_5}{2} \norm{\sigma(T)-\sigma_{\Omega}}_{\Hs}^2\\
        &+ \frac{\alpha_6}{2} \norm{\sqrt{\gamma(\varphi)}\epsilonu}_{L^2(Q)}^2 + \frac{\alpha_7}{2} \norm{z-z_Q}_{L^2(Q)}^2 \\
        &+  \alpha_8 \into z(T) \dx+ \frac{\alpha_9}{2}\norm{\cchi}_{L^2(Q)}^2. 
    \end{split}
    \end{equation}
The non-negative constants $\alpha_1, \dots, \alpha_9$ are weights that can not vanish all at the same time, while $\varphi_{Q}$, $\varphi_{\Omega}$, $\sigma_{Q}$, $\sigma_{\Omega}$, $z_{Q}$ are target functions. More explicitly,  the term $\varphi_{Q}$ (resp. $\sigma_{Q}$, $z_{Q}$) is a desired evolution for the tumor (resp. the lactate, the damage), while $\varphi_{\Omega}$ (resp. $\sigma_{\Omega}$) is a desired final configuration of the tumor (resp. concentration of the lactate). The third and 
eighth addends measure the size of the tumor and the magnitude of the damage at the end of the treatment. Since the presence of high mechanical stress, especially in certain areas of the brain, can compromise its functionality, we are interested in keeping it low. The sixth term serves this purpose, and $\gamma$ may be, for instance, the indicator function of a subdomain of $\Omega$ where the stress is intended to remain low. Finally, the last addend is a $L^2$ regularization for the controls.
We are interested in studying the minimizers of the cost functional \eqref{eq:cost_functional} subject to the PDE system \eqref{eq:problem}--\eqref{eq:initial_conditions} and constrained to a suitable admissible control set. We define it as $ \Uad = \Uad^1 \times \Uad^2$ with 
\begin{equation*}
    \begin{split}
        &\Uad^1 \coloneqq \{ \chi_1 \in L^2(\Vs) \cap L^{\infty}(Q) \, : \, \norm{\chi_1}_{L^2(\Vs)} \leq C_{\text{ad}}, \,\, \underline{\chi_1} \leq \chi_1 \leq \overline{\chi_1} \text{ a.e.}\}, \\
        &\Uad^2 \coloneqq \{ \chi_2 \in L^{\infty}(Q) \, : \, \underline{\chi_2} \leq \chi_2 \leq \overline{\chi_2} \text{ a.e.}\}, 
    \end{split}
\end{equation*}
where $C_{\text{ad}}>0$ is a fixed constant and $\underline{\chi_1}$, $\overline{\chi_1}$, $\underline{\chi_2}$, $\overline{\chi_2} \in L^{\infty}(Q)$ are given threshold functions such that $\Uad$ is nonempty. The admissible control set $\Uad$ is a subset of
\begin{equation*}
    \U= \U^1 \times \U^2 \coloneqq \left[ L^2(\Vs) \cap L^{\infty}(Q) \right] \times L^{\infty}(Q)
\end{equation*}
equipped with its natural norm that we denote with $\norm{\cdot}_{\U}$.

\begin{remark}
    Notice that $\Uad$ is a nonempty, closed, and convex subset of $\U$. Moreover, there exists a positive constant $R$ such that
    \begin{equation*}
        \Uad \subseteq \Ur \coloneqq \{ \cchi \in \U \, : \,  \norm{\cchi}_{\U} < R\}.
    \end{equation*}
\end{remark}

\noindent \textbf{Aim and plan of the paper.} The goal of this paper is to prove that the optimal control problem has at least one minimizer and to derive first-order necessary conditions for optimality. It is organized as follows. In \Cref{section:state_system}, after introducing some notation and preliminary results, we enlist the hypotheses we are going to use throughout the paper. Then, we recall some known results about the state system \eqref{eq:problem}--\eqref{eq:initial_conditions} from the previous work \cite{Cavalleri_Colli_Miranville_Rocca_2025} and we prove a strict separation property for the damage.  In \Cref{section:opt_control_problem} we properly state the optimal control problem and prove that it admits at least one minimizer. 
In \Cref{section:linearized_system} we analyze the linearized system, and in \Cref{section:diff_solution_op} we use it to prove the differentiability of the control-to-state operator. Finally, in \Cref{section:adj_system}, we study the adjoint system and derive the first-order necessary conditions for optimality.

\section{The state system}\label{section:state_system}

\subsection{Notation and preliminaries}
\textbf{Notation.} In what follows, for any real Banach space $X$ with dual space $X'$, we indicate its norm as $\norm{\cdot}_{X}$ and the dual pairing between $X'$ and $X$ as $\duality{\cdot}{\cdot}_{X}$. We denote the Lebesgue and Sobolev spaces over $\Omega$ as $L^p \coloneqq L^p(\Omega)$, $W^{k,p} \coloneqq W^{k,p}(\Omega)$ and $H^k \coloneqq W^{k,2}(\Omega)$, while for the Lebesgue spaces over $\Gamma$ we use $L^p_{\Gamma} \coloneqq L^p(\Gamma)$. 
For convenience, we set 
\begin{equation*}
    \Hs \coloneqq L^2
\end{equation*}
and we identify $\Hs$ with its dual space $\Hs'$. We introduce 
\begin{equation*}
    \Vs \coloneqq H^1, \qquad \Vs_0\coloneqq H^1_0,
\end{equation*}
where $H^1_0$ represents the set of $H^1$ functions with zero trace at the boundary. Additionally, we define
\begin{equation*}
    W \coloneqq \{v \in H^2 \, | \, \partial_{\nnu} v = 0 \text{ on } \Gamma\}, \qquad  \Wzs \coloneq H^2 \cap \Vs_0.
\end{equation*}
In both cases, the natural norm 
induced by $H^2$ is denoted by $\norm{\cdot}_{\Ws}$.
To simplify the notation, we do not always distinguish between scalar, vector, and matrix-valued spaces. For the sake of brevity, the norm of the Bochner space $W^{k,p}(0,T;X)$ is indicated as $\norm{\cdot}_{W^{k,p}(X)}$, omitting the time interval $(0,T)$. With the notation $C^0([0,T];X)$ we mean the space of continuous $X$-valued functions.
Finally, as is customary, $C$ represents a generic constant depending only on the problem's data and whose value might change from line to line or even within the same line. If we want to highlight a dependency on a certain parameter, we put it as a subscript (e.g., $C_{\tau}$ indicates a constant that depends on $\tau$, $C_0$ a constant that depends on the initial data, etc.). \\

\noindent \textbf{Useful inequalities.} We will make use of classical inequalities, such as H\"older, Young, Gronwall, Poincaré, and Poincaré--Wirtinger.
We will employ the following special cases of  Gagliardo--Nirenberg inequality (see \cite{Nirenberg_59}).
\begin{lemma}\label{lemma:special_case_gagliardo_nirenberg}
    Let $\Omega$ be a bounded Lipschitz domain in $\RR^d$. Then, it exists a constant $C$ such as, for every $v \in \Vs$, it holds
    \begin{align}
        &\norm{v}_{L^4} \leq  C \norm{v}^{\frac{1}{2}}_{\Hs}\norm{v}^{\frac{1}{2}}_{\Vs}  &&\text{if} \quad n=2,\\
        &\norm{v}_{L^3} \leq  C \norm{v}^{\frac{1}{2}}_{\Hs}\norm{v}^{\frac{1}{2}}_{\Vs}, \quad
        \norm{v}_{L^4} \leq  C \norm{v}^{\frac{1}{4}}_{\Hs}\norm{v}^{\frac{3}{4}}_{\Vs} &&\text{if} \quad  n=3.
    \end{align}
\end{lemma}

\noindent From this result, employing the Young inequality, it is easy to obtain an inequality that we will extensively use throughout the paper and, thus, is worth mentioning. For every $\delta > 0$ there exists a positive constant $C_{\delta}$ such that, for $p=3$ and $p=4$, the following holds:
\begin{equation}
    \norm{v}_{L^p}^2 \leq \delta \norm{v}_{\Hs}^2 + C_{\delta}\norm{\nabla v}_{\Hs}^2, 
\end{equation}
for every $v \in \Vs$, both in dimension $n=2$ and $n=3$.

\subsection{Hypotheses}
In this section, we enlist the hypotheses under which we will work throughout the whole paper. Most of them come from \cite{Cavalleri_Colli_Miranville_Rocca_2025} and are necessary to establish solutions' well-posedness, additional regularity, and continuous dependence. However, we need stronger assumptions to prove the results related to the optimal control problem. Specifically, the higher regularity requested for the nonlinearities is because we need to handle the corresponding terms in the linearized and adjoint systems.

\begin{enumerate}[(\rm{H\arabic*})]
\item \label{hyp:eq_phi}
We suppose that
\begin{align}
    &p, \, g \in C^{0,1}(\RR^2) \cap C^2(\RR^2),\\
    &0 \leq p \leq p^*, \quad 0 \leq g \leq g^*,
\end{align}
where $p^*$, $g^*$ denote positive constants, and that
\begin{equation}
    N \text{ is a positive constant.}
\end{equation}

\item \label{hyp:eq_lactate}
We assume that
\begin{align}
    &k_1, k_2, S \in C^{0,1}(\RR^2) \cap C^2(\RR^2) \text{, and }\\
    &0 \leq k_1 \leq k_1^*, \quad 0 < {k_2}_* \leq k_2 \leq k_2^*, \quad  0 \leq S \leq S^*,
\end{align}
where $k_1^*$, ${k_2}_*$, $k_2^*$, and $S^*$ are given constants.

\item \label{hyp:eq_displacement}
The  fourth-order tensors $\A = (a_{ijkh})$, $\B = (b_{ijkh}) : \RR^2 \to \RR^{n \times n \times n \times n}$ satisfy:
\begin{align}
     & \A \text{ is constant, symmetric, and strictly positive definite, and}\\
     & \B \in C^{0,1}(\RR)\cap C^2(\RR^2)\text{ is bounded, symmetric, and positive definite}.
\end{align}
Moreover, we assume
\begin{equation*}
\ff \in L^{\infty}(H).
\end{equation*}
\end{enumerate}

\noindent  In the following, being able to handle the maximal monotone operator $\beta$ and its higher derivatives will be essential. To this end, the key point is proving a strict separation property for the damage variable $z$. This leads us to adopt a more restrictive form of the convex potential $\hat{\beta}$, moving beyond the quite general hypotheses employed in the previous analysis \cite{Cavalleri_Colli_Miranville_Rocca_2025}.

\begin{enumerate}[(\rm{H\arabic*}),resume]
\item \label{hyp:beta}
We consider a $\widehat{\beta}: \RR \to [0,+\infty]$ such that 
\begin{align}
     &\widehat{\beta} \text{ is lower semicontinuous, convex, and}\\
     &\widehat{\beta} \in C^3(0,1).
\end{align}
Its derivative  $\beta \coloneqq \widehat{\beta}'$ satisfies the growth conditions:
\begin{equation}
    \lim_{r \to 0^+} \beta(r) = - \infty, \qquad \lim_{r \to 1^-} \beta(r) = + \infty.
\end{equation}

\item \label{hyp:pi}
We consider a function $\widehat{\pi} \in C^1(\RR)\cap C^3(0,1)$ and we denote by $\pi \coloneqq \widehat{\pi}'$ its derivative, requiring that
\begin{align}
    &\widehat{\pi} \text{  is concave,}\\
    &\pi \text{ is Lipschitz continuous.}
\end{align}
\end{enumerate}

\begin{ex}
    The prototypical example to keep in mind is the logarithmic potential
\begin{equation*}
    \hat{\beta}(r) + \hat{\pi}(r) = C_1 \left[r\ln{r} + (1-r)\ln{(1-r)}\right] - C_2r^2
\end{equation*}
for some given and positive constants $C_1,C_2$, which clearly fulfills the hypotheses \ref{hyp:beta} and \ref{hyp:pi}.
\end{ex}

\begin{enumerate}[(\rm{H\arabic*}),resume]
\item \label{hyp:iota}
We suppose that
\begin{equation}
    \iota \in L^{\infty}(Q) \cap H^1(0,T; \Hs).
\end{equation}

\item \label{hyp:psi}
We assume that $\Psi \in W^{2,\infty}(\Omega \times \RR \times \RR^{n \times n})$ and that
\begin{equation}\label{eq:psi_lipschitz}
    \begin{split}
        &\Psi(x, \cdot, \cdot) :  \RR \times \RR^{n \times n} \to \RR \text{ is Lipschitz continuous, i.e.,}\\
    & \exists C_{\Psi}>0 \text{ s.t. } |\Psi(x,\varphi_1,\epsilon_1)-\Psi(x,\varphi_2,\epsilon_2)| \leq C_{\Psi} \left(|\varphi_1-\varphi_2| + |\epsilon_1 - \epsilon_2|\right) 
    \end{split}
\end{equation}
for a.e. $x \in \Omega$, for all $\epsilon_1, \epsilon_2 \in \RR^{n \times n}$, $\varphi_1, \varphi_2 \in \RR$.
\end{enumerate}

\noindent In the following, for the sake of brevity, we will omit in the notation the dependence of $\Psi$ on the point $x$, using $\Psi(\varphi,\epsilon)$ instead of $\Psi(x,\varphi,\epsilon)$.

\begin{remark}
    We require the boundedness of $\Psi(\varphi,\epsilonu)$ to establish a separation property for the damage variable $z$. It is worth noting that this assumption is not as restrictive as it might initially appear. Indeed, we will prove the boundedness of $\varphi$, and since we are working within the framework of linear elasticity, $\epsilonu$ is expected to remain small. Consequently, under hypothesis \eqref{eq:psi_lipschitz}, the boundedness of $\Psi(\varphi,\epsilonu)$ would follow asking that $\Psi( 0, \mathbf{0}) \in L^{\infty}(\Omega)$. However, as we are unable to prove the boundedness of $\epsilonu$ mathematically, it becomes necessary to explicitly impose $\Psi$ the boundedness of as an additional condition. The boundedness of the higher derivatives is required to handle the linearized coefficients in the linearized and adjoint systems.
\end{remark}

\begin{enumerate}[(\rm{H\arabic*}),resume]
\item \label{hyp:boundary_conditions}
Regarding the boundary conditions, we suppose that
\begin{align}
    \sigma_{\Gamma} \in L^{2}(0,T; L^2_{\Gamma}),\quad 0 \leq \sigma_{\Gamma} \leq M_0,
\end{align}
where $M_0$ is a fixed positive constant.

\item \label{hyp:initial_conditions}
Regarding the initial conditions, we assume that
\begin{align}
    &\varphi_0 \in \Ws, \quad 0 \leq \varphi_0 \leq N,\\
    &\sigma_0 \in \Hs, \quad 0 \leq \sigma_0 \leq M_0, \label{eq:hp_sigma0}\\
    &\uu_0 \in \Wzs,\\
    &z_0 \in \Ws, \quad  0 < \text{ess inf}(z_0), \quad  \text{ess sup}(z_0) < 1.
\end{align}
\end{enumerate}

\begin{remark}
    Since the domain of such $\hat{\beta}$ is the physically meaningful interval $[0,1]$, the initial datum $z_0$ must be chosen such that its values lie within this interval. Recalling again that our goal is to prove a separation property for the damage, we request that $z_0$ stays bounded away from $0$ and $1$.
\end{remark}

\noindent Finally, regarding the cost functional $\mathcal{J}$, we make the following assumptions. 

\begin{enumerate}[(\rm{H\arabic*}), resume]
    \item \label{hyp:J_coeff} The coefficients $\alpha_1, \dots, \alpha_9$ are nonnegative constants that can not vanish all at the same time.
    \item \label{hyp:target_functions}The target functions satisfy
    \begin{equation}
        \varphi_Q, \sigma_Q, z_Q \in L^2(Q),\qquad \varphi_{\Omega}, \sigma_{\Omega} \in \Hs.
    \end{equation}
    \item \label{hyp:n} The coefficient $\gamma: \Omega \times \RR \to [0,+\infty)$ is a Carathéodory function such that 
    \begin{equation}
        \gamma(x,\cdot) \in C^1(\RR) 
    \end{equation}
    for a.e. $x \in \Omega$. Moreover, it exists a constant $C_{\gamma} >0$ such that
    \begin{equation}
        |\gamma(x,\varphi)| + |\partial_{\varphi} \gamma (x,\varphi)| \leq C_{\gamma}
    \end{equation}
    for a.e. $x \in \Omega$ and for all $\varphi \in [0,N]$. 
\end{enumerate}
Even if $\gamma$ also depends on the point $x \in \Omega$, in the following, we will employ the shorter notation $\gamma(\varphi)$ instead of $\gamma(x,\varphi)$. For the same reason, the partial derivative of $\gamma$ with respect to the variable $\varphi$ will be denoted by $\gamma'(\varphi)$. 

\subsection{Previous results}
\begin{defn}[State system] \label{defn:weak_solutio_state_problem}
    We say that the quadruplet $(\varphi,\, \sigma,\, \uu,\, z)$ is a weak solution to the state problem \eqref{eq:problem}--\eqref{eq:initial_conditions} if  
    \begin{gather*}
            \varphi \in   H^1(0,T;\Hs) \cap L^{\infty}(0,T; \Vs) \cap L^{2}(0,T;\Ws), \quad 0 \leq \varphi \leq N,\\
            \sigma \in H^1(0,T;\Vs') \cap L^{\infty}(0,T;\Hs) \cap L^2(0,T;\Vs), \quad 0 \leq \sigma \leq M,\\
            \uu \in W^{1,\infty}(0,T;\Vs_0),\\
            z \in   H^1(0,T;\Hs) \cap L^{\infty}(0,T; \Vs) \cap L^{2}(0,T;\Ws), \quad 0 < z < 1
    \end{gather*}
   where $M=M(M_0, S^*)$, with 
   \begin{gather*}
       \varphi(0)=\varphi_0, \quad \sigma(0)=\sigma_0, \quad  \uu(0)=\uu_0,  \quad z(0)=z_0
   \end{gather*}
   a.e. in $\Omega$, such that
    \begin{subequations}\label{eq:state_problem_eq_with_spaces}
        \begin{align}
        & \into \partial_t\varphi \eta +  \nabla \varphi \cdot \nabla \eta \dx = \into \bigg[(p(\sigma,z)-\chi_1)\varphi\left(1-\frac{\varphi}{N}\right) - \varphi g(\sigma,z)\bigg] \eta \dx, \label{eq:phi_with_spaces}\\
        &\begin{aligned}
            &\duality{\partial_t\sigma}{\eta }_{\Vs} + \into \nabla \sigma \cdot \nabla \eta \dx + \into \frac{k_1(\varphi,z)\sigma}{k_2(\varphi,z) + \sigma}\eta\dx + \intg (\sigma - \sigmag) \eta \dA\\
            & \hspace{49.5ex}= \into \chi_2 S(\varphi,z) \eta \dx,
        \end{aligned}\label{eq:sigma_with_spaces}\\
        &  \into  \left[\A \epsilonut + \B(\varphi,z)\epsilonu  \right] : \varepsilon(\ttheta) \dx = \into \ff \cdot \ttheta \dx,\label{eq:u_with_spaces}\\
        & \into \left[\partial_t z \eta + \nabla z \cdot \nabla \eta  + \beta(z) \eta  + \pi(z) \eta \right] \dx = \into \left[ \iota -  \Psi(\varphi,\epsilonu) \right] \eta \dx,\label{eq:z_with_spaces}
    \end{align}
\end{subequations}
a.e. in $(0,T)$, for every $\eta \in \Vs$ and $\ttheta \in \Vs_0$.
\end{defn}

\noindent Since $\cchi$ is fixed and regular enough, well-posedness and continuous dependence can be proved as in \cite{Cavalleri_Colli_Miranville_Rocca_2025}. Note that $z$ has values in $(0,1)$ because $z$ belongs to the domain of $\beta$. The precise statement is contained in the theorem below.

\begin{thm}\label{thm:wellposedness}
    The following statements hold.
    \begin{enumerate}[(i)]
        \item For every $\chi \in \U$, the state system admits a unique solution in the sense of \Cref{defn:weak_solutio_state_problem}, with the following additional regularity
        \begin{gather*}
        \varphi \in H^1(0,T;\Vs) \cap L^{\infty}(0,T;\Ws) \cap L^2(0,T;H^3),\\
         \uu \in W^{1,\infty}(0,T;\Wzs),\\
         z \in H^1(0,T;\Vs) \cap L^{\infty}(0,T;\Ws).
    \end{gather*}
        \item For every $\chi \in \Ur$ and its associated solution to the state system $(\varphi,\sigma,\uu,z)$, the following estimate is satisfied
        \begin{equation}
            \begin{split}
                \norm{\varphi}_{H^1(\Vs) \cap L^{\infty}(\Ws) \cap L^{2}(H^3)} &+ \norm{\sigma}_{ H^1(\Vs') \cap L^{\infty}(\Hs) \cap L^2(\Vs)}\\
            &+ \norm{\uu}_{W^{1,\infty}(\Wzs)} + \norm{z}_{H^1(\Vs) \cap L^{\infty}(\Ws)} \leq C_R
            \end{split}
        \end{equation}
        for a positive constant $C_R$ that depends on $R$, the initial data, and the assigned functions, but not on $\cchi$.  \label{thm:wellposedness_energy_estimate}
        \item  For every couple of controls $\cchi, \overline{\cchi} \in \Ur$ with associated solutions to the state system given respectively by $(\varphi,\sigma,\uu,z)$ and $(\overline{\varphi},\overline{\sigma},\overline{\uu},\overline{z})$, the following continuous dependence inequality is satisfied
        \begin{equation*}
        \begin{split}
            \norm{\varphi-\overline{\varphi}}_{L^{\infty}(\Hs)\cap L^2(\Vs)} + \norm{\sigma-\overline{\sigma}}_{L^{\infty}(\Hs)\cap L^2(\Vs)}+ \norm{\uu-\overline{\uu}}&_ {H^1(\Vs_0)}\\
            +  \norm{z-\overline{z}}_{L^{\infty}(\Hs)\cap L^2(\Vs)} &\leq C_R\norm{\cchi - \overline{\cchi}}_{L^2(Q)}
        \end{split}
    \end{equation*}
    for a certain positive constant $C_R$ that depends on $R$, the initial data, and the assigned functions, but not on $\cchi, \overline{\cchi}$. \label{thm:wellposedness_continuous_dependence}
    \end{enumerate}
\end{thm}

\begin{remark} \label{remark:add_continuous_estimate}
    From \ref{thm:wellposedness_energy_estimate} and \ref{thm:wellposedness_continuous_dependence} in \Cref{thm:wellposedness}, employing standard interpolation results, we are able to prove the following continuity estimates which we will need later on. Precisely, for every $\cchi,\overline{\cchi}$ in $\Ur$, we have:
    \begin{equation}\label{eq:add_continuous_estimate_inf4}
        \norm{\varphi-\overline{\varphi}}_{L^{\infty}(L^4)} + \norm{\epsilonu - \varepsilon(\overline{\uu})}_{L^{\infty}(L^4)} + \norm{z-\overline{z}}_{L^{\infty}(L^4)} \leq C_R \norm{\cchi - \overline{\cchi}}_{L^2(Q)}^{\frac{1}{4}}.
    \end{equation} 
    Applying Gagliardo--Nirenberg interpolation inequality from \Cref{lemma:special_case_gagliardo_nirenberg}, estimates \ref{thm:wellposedness_energy_estimate} and \ref{thm:wellposedness_continuous_dependence} from the Theorem above, we have:
    \begin{equation*}
        \begin{split}
             \norm{&\epsilonu - \varepsilon(\overline{\uu})}_{L^{\infty}(L^4)} \leq C  \norm{\epsilonu - \varepsilon(\overline{\uu})}_{L^{\infty}(\Hs)}^{\frac{1}{4}}\norm{\epsilonu - \varepsilon(\overline{\uu})}_{L^{\infty}(\Vs)}^{\frac{3}{4}}\\
             &\leq  C\norm{\uu - \overline{\uu}}_{L^{\infty}(\Vs_0)}^{\frac{1}{4}}\norm{\uu - \overline{\uu}}_{L^{\infty}(\Wzs)}^{\frac{3}{4}} \leq C_R \norm{\uu-\overline{\uu}}_{H^1(\Vs_0)}^{\frac{1}{4}} \leq C_R \norm{\cchi - \overline{\cchi}}_{L^2(Q)}^{\frac{1}{4}}.
        \end{split}
    \end{equation*}
    The same can be performed for $\varphi$ and $z$. Similarly, we have
    \begin{equation}\label{eq:add_continuous_estimate_4inf}
        \norm{\varphi-\overline{\varphi}}_{L^4(L^{\infty})}  + \norm{z-\overline{z}}_{L^4(L^{\infty})} \leq C_R \norm{\cchi - \overline{\cchi}}_{L^2(Q)}^{\frac{1}{4}}.
    \end{equation}
    In fact, applying Gagliardo--Nirenberg inequality and the embedding $W^{1,4} \hookrightarrow L^{\infty}(\Omega)$, we obtain
    \begin{equation*}
        \norm{\varphi - \overline{\varphi}}_{L^{\infty}(\Omega)} \leq \norm{\varphi - \overline{\varphi}}_{W^{1,4}}\leq C \norm{\varphi - \overline{\varphi}}_{\Vs}^{\frac{1}{4}} \norm{\varphi - \overline{\varphi}}_{\Ws}^{\frac{3}{4}} \leq C_R \norm{\varphi - \overline{\varphi}}_{\Vs}^{\frac{1}{4}}, 
    \end{equation*}
    where the last inequality follows from the energy estimate in \ref{thm:wellposedness_energy_estimate}.
    Elevating to the power $4$th and integrating in time over $(0,T)$ leads to 
    \begin{equation*}
         \begin{split}
             \norm{\varphi - \overline{\varphi}}_{L^4(L^{\infty})}^4 &= \intT \norm{\varphi - \overline{\varphi}}_{L^{\infty}(\Omega)}^4 \dt \leq C_R \intT \norm{\varphi - \overline{\varphi}}_{\Vs} \dt\\
             &= C_R \norm{\varphi - \overline{\varphi}}_{L^1(\Vs)}\leq C_R \norm{\varphi - \overline{\varphi}}_{L^2(\Vs)} \leq C_R \norm{\cchi-\overline{\cchi}}_{L^2(Q)}
         \end{split}
    \end{equation*}
    thanks to the continuous estimate \ref{thm:wellposedness_continuous_dependence}. The same holds for the damage. Finally, the following inequality is satisfied
    \begin{equation}\label{eq:add_continuous_estimate_43}
        \norm{\varphi-\overline{\varphi}}_{L^4(L^3)}  + \norm{z-\overline{z}}_{L^4(L^3)} \leq C_R \norm{\cchi - \overline{\cchi}}_{L^2(Q)}.
    \end{equation}
    Proceeding as before, 
    \begin{equation*}
        \begin{split}
             \norm{\varphi-\overline{\varphi}}_{L^4(L^3)}^4 &= \intT \norm{\varphi-\overline{\varphi}}_{L^3}^4 \dt \leq C \intT \norm{\varphi-\overline{\varphi}}_{\Hs}^2  \norm{\varphi-\overline{\varphi}}_{\Vs}^2 \dt\\
             &\leq C \norm{\varphi-\overline{\varphi}}_{L^{\infty}(\Hs)}^2 \intT  \norm{\varphi-\overline{\varphi}}_{\Vs}^2 \dt\\
             & = C  \norm{\varphi-\overline{\varphi}}_{L^{\infty}(\Hs)}^2  \norm{\varphi-\overline{\varphi}}_{L^2(\Vs)}^2 \leq C_R \norm{\cchi - \overline{\cchi}}_{L^2(Q)}^4,
        \end{split}
    \end{equation*}
    and the same for $z$.
\end{remark}

\subsection{A strict separation property for the damage}
\noindent Differently from \cite{Cavalleri_Colli_Miranville_Rocca_2025}, we impose more restrictive assumptions on the potential $\hat{\beta}$, as well as boundedness for $\iota$ and $\Psi$, which enable us to establish a separation property for the damage $z$.
\begin{prop}\label{prop:separation_property}
    There exist $0 < r_* \leq r^* < 1$ which may depend on the data of the problem and on $R$ such that, for every $\cchi \in \Ur$ with the associated solution to the state problem $(\varphi,\sigma,\uu,z)$,
    \begin{equation}
        r_* \leq z \leq r^*
    \end{equation}
    a.e. in $Q$.
\end{prop}
\begin{proof}
    We claim that there exist $0 < r_* \leq r^* <1$ such that
    \begin{enumerate}
        \item $r_* \leq \text{ess inf}(z_0)$,\label{eq:inf_sep}
        \item $\text{ess sup}(z_0) \leq r^*$, \label{eq:sup_sep}
        \item $\beta(r)+\pi(r) + \norm{\iota}_{L^{\infty}}+ \norm{\Psi}_{L^{\infty}} \leq 0$ for all $r \in (0,r_*)$, \label{eq:sign_ineq_sep_0}
        \item $\beta(r)+\pi(r) - \norm{\iota}_{L^{\infty}} -  \norm{\Psi}_{L^{\infty}} \geq 0$ for all $r \in (r^*,1)$.\label{eq:sign_ineq_sep_1}
    \end{enumerate}
    We can find such $r_*,r^*$ satisfying conditions \ref{eq:inf_sep} and \ref{eq:sup_sep} because of hypothesis \ref{hyp:initial_conditions} ensuring that $z_0$ remains bounded away from $0$ and $1$. Moreover, regarding points \ref{eq:sign_ineq_sep_0} and \ref{eq:sign_ineq_sep_1}, we recall that, from the specific choice we made for the potential $\hat{\beta}$, it holds
    \begin{equation*}
        \beta(r) \to - \infty \text{ if } r \to 0^+, \qquad  \beta(r) \to + \infty \text{ if } r \to 1^-
    \end{equation*}
    where, from hypothesis \ref{hyp:pi}, $\pi \in C^0(\RR)$ so is bounded over $[0,1]$ and, from hypotheses \ref{hyp:iota} and \ref{hyp:psi}, $\iota$ and $\Psi$ are bounded.
    Let $\cchi$ be an arbitrary control in $\Ur$ and $\S(\cchi) = (\varphi,\sigma,\uu,z)$. We test the equation \eqref{eq:z_with_spaces} with $(z-r^*)^+$, obtaining:
    \begin{equation*}
        \begin{split}
            0 &= \frac{1}{2} \frac{\text{d}}{\text{d}t} \into |(z-r^*)^+|^2 \dx
            + \into |\nabla (z-r^*)^+|^2 \dx\\
            &\quad + \into \left[\beta(z)+\pi(z)- \iota +\Psi(\varphi,\epsilonu)\right](z-r^*)^+ \dx\\
            &\geq \frac{1}{2} \frac{\text{d}}{\text{d}t} \into |(z-r^*)^+|^2 \dx, 
        \end{split}
    \end{equation*}
    where we used the property \ref{eq:sign_ineq_sep_1} in order to obtain the inequality. Integrating in time over the interval $(0,t)$, it follows that
    \begin{equation*}
        \into |(z-r^*)^+|^2 \dx \leq \into |(z_0-r^*)^+|^2 \dx = 0
    \end{equation*}
    because, accordingly to property \ref{eq:sup_sep}, $z_0$ is smaller or equal to $r^*$ almost everywhere, thus $(z_0-r^*)^+=0$. This proves that $(z-r^*)^+=0$ and, equivalently, that $z$ is smaller or equal to $r^*$ almost everywhere in $Q$. Proceeding in the same way, we test equation \eqref{eq:z_with_spaces} with $-(z-r_*)^-$. Employing inequality \ref{eq:sign_ineq_sep_0} and then integrating in time and using \ref{eq:inf_sep}, we obtain that $(z-r_*)^-$ is equal to $0$. Thus, we have that $z$ is bigger or equal to $r_*$ almost everywhere in $Q$.
\end{proof}

\begin{remark}
    Let us highlight the fact that thanks to the strict separation property we have just proved, from now on we will treat $\beta$ as a regular potential. Moreover, we trivially deduce that 
    \begin{equation*}
        \norm{\beta(z)}_{L^{\infty}} +\norm{\beta'(z)}_{L^{\infty}} +\norm{\beta''(z)}_{L^{\infty}} \leq C_R,
    \end{equation*}
    since $\beta \in C^2(0,1)$ by hypothesis \ref{hyp:beta}.
\end{remark}

\section{The optimal control problem} \label{section:opt_control_problem}
 In view of the previous \Cref{thm:wellposedness}, we introduce the so-called control-to-state operator or solution operator, which maps every control $\cchi$ to the unique solution to the associated state problem. More precisely, we introduce the state-space
\begin{multline*}
    \VS \coloneqq \left[H^1(0,T;\Hs) \cap L^{\infty}(0,T;\Vs) \cap L^2(0,T;\Ws)\right]\\
    \times \left[H^1(0,T;\Vs') \cap L^{\infty}(0,T;\Hs) \cap L^2(0,T;\Vs)\right] \times W^{1,\infty}(0,T;\Ws) \\
    \times \left[H^1(0,T;\Hs) \cap L^{\infty}(0,T;\Vs) \cap L^2(0,T;\Ws)\right] 
\end{multline*}
which is continuously embedded into the larger
\begin{multline*}
    \V \coloneqq \left[C^0([0,T];\Hs) \cap L^2(0,T;\Vs)\right] \times \left[C^0([0,T];\Hs) \cap L^2(0,T;\Vs)\right]\\ \times H^1(0,T;\Vs_0) \times \left[C^0([0,T];\Hs) \cap L^2(0,T;\Vs)\right].
\end{multline*}
We define
\begin{align}
    \S : \U \to \V,\qquad 
    \cchi \mapsto (\varphi,\sigma,\uu,z).
\end{align}
Notice that $\S$ is well-defined over $\U$ and Lipschitz continuous over $\Ur$.  We introduce the reduced cost functional as
\begin{equation}
    \J(\cchi) \coloneqq \mathcal{J}(\S(\cchi),\cchi).
\end{equation}
Then, the optimal control problem can be stated as
\begin{equation}\label{eq:optimal_control_problem}
        \min_{\cchi \in \Uad} \J(\cchi),
\end{equation}
which means that we search a minimizer for the functional $\mathcal{J}$ subject to the PDE system \eqref{eq:problem}--\eqref{eq:initial_conditions} and constrained to $\Uad$. 

\begin{thm}
    There exists at least one minimizer $\cchi^* \in \Uad$ to the optimal control problem \eqref{eq:optimal_control_problem}. 
\end{thm}

\begin{proof} 
    The reduced cost functional is proper and non-negative, so $\inf_{\cchi \in \Uad}\J(\cchi)$ is finite and non-negative. Let $\{\cchi_n\}_{n \in \NN} \subseteq \Uad$ be a minimizing sequence for $\J$, meaning that
    \begin{equation*}
        \inf_{\cchi \in \Uad} \J(\cchi) = \lim_{n\to+\infty} \J(\cchi_n).
    \end{equation*}
    We denote the corresponding solution to the state system as $(\varphi_n,\sigma_n,\uu_n,z_n) = \S(\cchi_n)$. Since the sequence $\{\cchi_n\}_{n\in\NN}\in \Uad$, it is uniformly bounded in $\U$ and, consequentially, there exists a $\cchi^* \in \U$ such that, along a subsequence that we do not relabel, 
    \begin{alignat*}{2}
        \chi_{n,1} \to \chi_1^* &\qquad \text{weakly-}\ast &&\qquad \text{in } L^2(0,T;\Vs)\cap L^{\infty}(Q),\\
        \chi_{n,2} \to \chi_2^* &\qquad \text{weakly-}\ast &&\qquad \text{in } L^{\infty}(Q).
    \end{alignat*}
    Notice that $\Uad$ is convex and closed in the space $L^2(0,T;\Vs) \times L^2(Q)$ thus, it is sequentially weakly closed. This justifies the fact that the limit $\cchi^*$ belongs to $\Uad$.
    Moreover, by the uniform boundedness of the solution sequence from \ref{thm:wellposedness_energy_estimate} in \Cref{thm:wellposedness}, applying Banach--Alaouglu (see, e.g., \cite{brezis2011}) and Aubin--Lions Theorems (see \cite[][Section 8, Corollary 4]{Simon_86}), there exists a $(\varphi^*,\sigma^*,\uu^*,z^*) \in \VS$ such that, along a further subsequence,
    \begin{alignat*}{2}
         \varphi_n \to \varphi^* &\qquad \text{weakly-}\ast &&\qquad\text{in }  H^{1}(0,T;\Hs) \cap L^{\infty}(0,T;\Vs) \cap L^2(0,T;\Ws),\\
         &\qquad \text{strong} &&\qquad\text{in } L^2(0,T;\Vs),\\
         &\qquad \text{a.e.} &&\qquad\text{in } Q,\\
         \sigma_n \to \sigma^* &\qquad \text{weakly-}\ast &&\qquad \text{in }  H^{1}(0,T;\Vs') \cap L^{\infty}(0,T;\Hs) \cap L^2(0,T;\Vs),\\
         &\qquad \text{strong} &&\qquad\text{in } L^2(0,T;\Hs),\\
         &\qquad \text{a.e.} &&\qquad\text{in } Q,\\
         \uu_n \to \uu^* &\qquad \text{weakly-}\ast &&\qquad\text{in }  W^{1,\infty}(0,T;\Wzs),\\
         &\qquad \text{strong} &&\qquad\text{in } C^0([0,T];\Vs_0),\\
         z_n \to z^* &\qquad \text{weakly-}\ast &&\qquad \text{in }  H^{1}(0,T;\Hs) \cap L^{\infty}(0,T;\Vs) \cap L^2(0,T;\Ws),\\
         &\qquad \text{strong} &&\qquad\text{in } L^2(0,T;\Vs),\\
         &\qquad \text{a.e.} &&\qquad\text{in } Q.
    \end{alignat*}
The first step of the proof consists of proving that $\S(\cchi^*)=(\varphi^*,\sigma^*,\uu^*,z^*)$. To do so, thanks to the convergences above, we can pass to the limit in the PDE system satisfied by $(\varphi_n,\sigma_n,\uu_n,z_n)$ and $\cchi_n$, from which $(\varphi^*,\sigma^*,\uu^*,z^*)$ is the unique solution to the state system associated with $\cchi^*$. The second step is showing that 
\begin{equation*}
    \inf_{\cchi \in \Uad} \J(\cchi) = \J(\cchi^*).
\end{equation*}
To this end, we write the reduced cost functional as the sum of the following terms
\begin{align*}
     \J_1(\cchi) &= \frac{\alpha_1}{2} \norm{\varphi-\varphi_Q}_{L^2(Q)}^2 + \frac{\alpha_4}{2} \norm{\sigma-\sigma_Q}_{L^2(Q)}^2+ \frac{\alpha_7}{2} \norm{z-z_Q}_{L^2(Q)}^2 + \frac{\alpha_9}{2}\norm{\cchi}_{L^2(Q)}^2,\\
     \J_2(\cchi) & =  \frac{\alpha_2}{2} \norm{\varphi(T)-\varphi_{\Omega}}_{\Hs}^2+ \alpha_3 \norm{\varphi(T)}_{L^1(\Omega)} \\
     &\quad +\frac{\alpha_5}{2} \norm{\sigma(T)-\sigma_{\Omega}}_{\Hs}^2 + \frac{\alpha_6}{2} \norm{\sqrt{\gamma(\varphi)}\epsilonu}_{L^2(Q)}^2  +  \alpha_8 \norm{z(T)}_{L^1(\Omega)}.
\end{align*}
By the weak convergences above and weak lower semicontinuity of the $L^2$-norm, we obtain
\begin{equation}
    \liminf_{n \to +\infty}{\J_1(\cchi_n)} \geq \J_1(\cchi^*).
\end{equation}

\noindent Notice that, by uniform boundedness of the sequence from \ref{thm:wellposedness_energy_estimate} in \Cref{thm:wellposedness}, and Aubin--Lions Theorem (see \cite[][Section 8, Corollary 4]{Simon_86}), we have the strong convergences
\begin{alignat}{2}
         \varphi_n \to \varphi^* &\qquad \text{strongly} &&\qquad\text{in }  C^0([0,T];\Hs),\label{eq:strong_conv_phi_opt_control}\\
          \sigma_n \to \sigma^* &\qquad \text{strongly} &&\qquad\text{in }  C^0([0,T];\Vs'),\label{eq:strong_conv_sigma_opt_control}\\
         z_n \to z^* &\qquad \text{strongly} &&\qquad\text{in }  C^0([0,T];\Hs).\label{eq:strong_conv_z_opt_control}
\end{alignat}
\noindent Again from  \ref{thm:wellposedness_energy_estimate} in \Cref{thm:wellposedness}, $\{\sigma^*(T)\}_{n \in \NN}$ is uniformly bounded in $\Hs$. Thus, extracting a further subsequence,   
\begin{alignat}{2}
    \sigma_n(T) \to \sigma^*(T) &\qquad \text{weakly} &&\qquad \text{in } \Hs,\label{eq:weak_conv_sigmaT_opt_control}
\end{alignat}
 where we are able to identify the limit with $\sigma^*(T)$ because of convergence \eqref{eq:strong_conv_sigma_opt_control}.
 Since $\varphi_n \to \varphi^*$ a.e. and  $\gamma$ is bounded by hypothesis \ref{hyp:n},
\begin{alignat}{2}
    \sqrt{\gamma(\varphi_n)} \eta \to \sqrt{\gamma(\varphi^*)}\eta &\qquad \text{strongly} &&\qquad \text{in } L^2(0,T;\Hs), \label{eq:strong_conv_n_opt_control}
\end{alignat}
for every $\eta \in L^2(0,T;\Hs)$ thanks to Dominated Convergence Theorem. Moreover, from the convergences enlisted above, we know that
\begin{alignat}{2}
    \varepsilon(\uu_n) \to \varepsilon(\uu^*) &\qquad \text{weakly} &&\qquad \text{in } L^2(0,T;\Hs).\label{eq:weak_conv_epsu_opt_control}
\end{alignat}
Putting together \eqref{eq:strong_conv_n_opt_control} and \eqref{eq:weak_conv_epsu_opt_control}, we deduce that
\begin{alignat}{2}
    \sqrt{\gamma(\varphi_n)}\varepsilon(\uu_n) \to \sqrt{\gamma(\varphi^*)}\varepsilon(\uu^*) &\qquad \text{weakly} &&\qquad \text{in } L^2(0,T;\Hs).\label{eq:weak_conv_u_opt_control}
\end{alignat}
By the weak lower semicontinuity of the $L^2$-norm and the weak convergences \eqref{eq:weak_conv_sigmaT_opt_control}, \eqref{eq:weak_conv_u_opt_control}, and by the strong continuity of the $L^2$-norm and the strong convergences \eqref{eq:strong_conv_phi_opt_control}, \eqref{eq:strong_conv_z_opt_control}, we have
\begin{equation*}
    \liminf_{n\to+\infty}\J_2(\cchi_n) \geq \J_2(\cchi^*)
\end{equation*}
from which the thesis follows.
\end{proof}

\noindent We aim to establish first-order optimality conditions for the optimal control problem \eqref{eq:optimal_control_problem}. To do so, the standard procedure is to prove the Fréchet differentiability of the control-to-state operator. With this purpose, we linearize the state system. 

\section{The linearized state system}\label{section:linearized_system}
We consider a fixed control $\cchi \in \Uad$ with the associated state given by $(\varphi,\sigma,\uu,z) = \S(\cchi)$. For every small $\hh = (h_1,h_2) \in \U$, we consider the perturbed variables
\begin{equation}
    \varphi + \xi, \quad \sigma + \rho, \quad \uu + \oomega, \quad z + \zeta, \quad \cchi + \hh,
\end{equation}
and the corresponding state system. Linearizing it near $(\varphi,\sigma,\uu,z)$, $\cchi$ and approximating the nonlinearities with their first order Taylors'expantions, we have that $(\xi,\rho,\oomega,\zeta)$, $\hh$ satisfy the linear PDE system
\begin{subequations}\label{eq:linearized_problem}
    \begin{align}
        & \partial_t\xi - \Delta \xi = a_1 \xi + a_2 \rho + a_3 \zeta + a_4 h_1 , \label{eq:lin_phi}\\
        & \partial_t\rho - \Delta \rho = b_1 \xi + b_2 \rho + b_3 \zeta + b_4 h_2,\label{eq:lin_lactate}\\
        &  - \diver\left[\A \epsilonot + \B(\varphi,z)\epsilono  \right] = - \diver\left[ \cc_1 \xi + \cc_2 \zeta \right], \label{eq:lin_displacement}\\
        & \partial_t \zeta - \Delta \zeta =  d_1 \xi + \dd_2 : \epsilono + d_3 \zeta, \label{eq:lin_damage}
    \end{align}
\end{subequations}
where, for the sake of better readability, we introduced the following notation:
\begin{subequations}
\allowdisplaybreaks
    \begin{align*}
        a_1 &= U_{,\varphi} = (p(\sigma,z)- \chi_1) \left(1-2\frac{\varphi}{N}\right) -  g(\sigma,z),\\
        a_2 &= U_{,\sigma} =  p_{,\sigma}(\sigma,z) \varphi \left(1-\frac{\varphi}{N}\right) - \varphi g_{,\sigma}(\sigma,z),\\
        a_3 &= U_{,z} = p_{,z}(\sigma,z) \varphi \left(1-\frac{\varphi}{N}\right) - \varphi g_{,z}(\sigma,z),\\
        a_4 &= U_{,\chi_1} = -\varphi\left(1 - \frac{\varphi}{N}\right),\\
        b_1 &= -K_{,\varphi}(\varphi,\sigma,z) +\chi_2 S_{,\varphi}(\varphi,z)\\
            &= - \frac{k_{1,\varphi}(\varphi,z)\sigma}{k_2(\varphi,z)+\sigma} + \frac{k_1(\varphi,z)\sigma k_{2,\varphi}(\varphi,z)}{(k_2(\varphi,z)+\sigma)^2} + \chi_2 S_{,\varphi}(\varphi,z),\\
        b_2 &= -K_{,\sigma}(\varphi,\sigma,z)= - \frac{k_1(\varphi,z)}{k_2(\varphi,z)+\sigma} + \frac{k_1(\varphi,z)\sigma}{(k_2(\varphi,z) + \sigma)^2},\\
        b_3 &= -K_{,z}(\varphi,\sigma,z) +\chi_2 S_{,z}(\varphi,z)\\
            &=- \frac{k_{1,z}(\varphi,z)\sigma}{k_2(\varphi,z)+\sigma} + \frac{k_1(\varphi,z)\sigma k_{2,z}(\varphi,z)}{(k_2(\varphi,z)+\sigma)^2} + \chi_2 S_{,z}(\varphi,z),\\
        b_4 & = S(\varphi,z),\\
        \cc_1 & =- \B_{,\varphi}(\varphi,z)\epsilonu,  \\
        \cc_2 & = - \B_{,z}(\varphi,z)\epsilonu,\\
        d_1 &= -\Psi_{,\varphi}(\varphi,\epsilonu),\\
        \dd_2 &= -\Psi_{,\varepsilon}(\varphi,\epsilonu),\\
        d_3 &= - \beta'(z) -\pi'(z).
    \end{align*}
\end{subequations}

\begin{remark}\label{remark:boundedness_coeff_lin}
    Notice that, since all the assigned functions are Lipschitz continuous and bounded, and because of the regularity we have already proved for $(\varphi,\sigma,\uu,z)$ in \Cref{thm:wellposedness}, 
    \begin{equation*}
        a_1, \dots, a_4, b_1, \dots, b_4, d_1, \dd_2 \in L^{\infty}(Q), 
    \end{equation*}
    and they are uniformly bounded by a constant that depends on $R$. Moreover, 
    \begin{equation*}
        \cc_1, \cc_2 \in L^{\infty}(0,T;L^p)
    \end{equation*}
    for any $p \in [1,6]$ and their norm is uniformly bounded by a certain $C_R$, because 
    \begin{equation*}
        |\cc_1| + |\cc_2| \leq C |\epsilonu|,
    \end{equation*}
    and $\uu \in W^{1,\infty}(0,T; \Ws) \hookrightarrow W^{1,\infty}(0,T; W^{1,p})$. Regarding the last term, thanks to the regularity of $\beta$ in its domain and to the separation property we proved in \Cref{prop:separation_property}, 
    \begin{equation*}
        d_3 \in L^{\infty}(Q),
    \end{equation*}
    and its norm is bounded by a constant that depends on $R$.
\end{remark}
\noindent We couple the system \eqref{eq:linearized_problem} with the following boundary and initial conditions
\begin{align}
    \partial_{\nnu}\xi = \partial_{\nnu}\zeta = 0 \qquad \partial_{\nnu}\rho = - \rho, \qquad \oomega = 0 \qquad \text{on } \Sigma, \label{eq:linearized_boundary_cond}\\
    \xi(0) = \rho(0) = \zeta(0) = 0, \qquad \oomega(0) = \mathbf{0} \qquad \text{in } \Omega. \label{eq:linearized_initial_cond}
\end{align}

\noindent Notice that, even if for the formal derivation of the linearized system we started from a small perturbation $\hh \in \Uad$, the obtained system  \eqref{eq:linearized_problem}--\eqref{eq:linearized_initial_cond} makes sense for every $\hh \in L^2(Q) \times L^2(Q)$ .

\begin{prop} \label{prop:existence_lin}
     For every $\cchi \in \Ur$ with associated $\S(\cchi) = (\varphi,\sigma,\uu,z) \in \V$ and for every $\hh \in L^2(Q) \times L^2(Q)$ there exists a unique solution $(\xi, \rho, \oomega, \zeta)$ to the linearized state system \eqref{eq:linearized_problem}--\eqref{eq:linearized_initial_cond} in the sense that
    \begin{gather*}
            \xi \in   H^1(0,T;\Hs) \cap L^{\infty}(0,T; \Vs) \cap L^{2}(0,T;\Ws),\\
            \rho \in H^1(0,T;\Vs') \cap L^{\infty}(0,T;\Hs) \cap L^2(0,T;\Vs), \\
            \oomega \in W^{1,\infty}(0,T;\Vs_0),\\
            \zeta \in   H^1(0,T;\Hs) \cap L^{\infty}(0,T; \Vs) \cap L^{2}(0,T;\Ws) 
    \end{gather*}
    with 
   \begin{gather*}
        \xi(0) = \rho(0) = \zeta(0) = 0, \qquad \oomega(0) = \mathbf{0}
   \end{gather*}
   such that
   \begin{subequations}\label{eq:lin_problem_in_eq_with_spaces}
        \begin{align}
        & \into \partial_t\xi \eta \dx + \into \nabla \xi \cdot \nabla \eta \dx = \into [ a_1 \xi + a_2 \rho + a_3 \zeta + a_4 h_1] \eta \dx, \label{eq:phi_lin_with_spaces}\\
        &  \duality{\partial_t\rho}{\eta }_{\Vs} + \into \nabla \rho \cdot \nabla \eta \dx +  \intg \rho \eta \dA= \into [b_1 \xi + b_2 \rho + b_3 \zeta + b_4 h_2] \eta \dx,\label{eq:sigma_lin_with_spaces}\\
        &  \into  \left[\A \epsilonot + \B(\varphi,z)\epsilono  \right] : \varepsilon(\ttheta) \dx = \into  [\cc_1 \xi + \cc_2 \zeta ] : \varepsilon(\ttheta) \dx,\label{eq:u_lin_with_spaces}\\
        & \into \partial_t \zeta \eta \dx + \into \nabla \zeta \cdot \nabla \eta \dx = \into \left[ d_1 \xi + \dd_2 : \epsilono + d_3 \zeta \right] \eta \dx,\label{eq:z_lin_with_spaces}
    \end{align}
\end{subequations}
a.e. in $(0,T)$, for every $\eta \in \Vs$ and $\ttheta \in \Vs_0$. Moreover, the solution satisfies the following estimate:
\begin{multline}\label{eq:lin_sys_estimate}
    \norm{\xi}_{H^1(\Hs) \cap L^{\infty}(\Vs) \cap L^2(\Ws)} + \norm{\rho}_{H^1(\Hs) \cap L^{\infty}(\Vs) \cap L^2(\Ws)} +
    \norm{\oomega}_{W^{1,\infty}(\Vs_0)}\\
    + \norm{\xi}_{H^1(\Hs) \cap L^{\infty}(\Vs) \cap L^2(\Ws)} \leq C_R \norm{\hh}_{L^2(Q)^2}.
\end{multline}
\end{prop}
\begin{proof}
 Existence can be proved using a Galerkin scheme. Since it is a standard procedure, we will show only the formal a priori estimates that are necessary to pass to the limit from the discrete to the continuous system. We test \eqref{eq:lin_phi} with $\xi$, \eqref{eq:lin_lactate} with $\rho$, \eqref{eq:lin_damage} with $\zeta$, and sum the three equations. Applying the Young inequality and \Cref{remark:boundedness_coeff_lin}, we obtain
\begin{equation}\label{eq:enqergy_ineq1_lin}
    \begin{split}
        \frac{\text{d}}{\text{d}t}& \left( \norm{\xi}_{\Hs}^2 + \norm{ \rho}_{\Hs}^2 + \norm{\zeta}_{\Hs}^2\right) + \norm{\nabla \xi}_{\Hs}^2 + \norm{\nabla \rho}_{\Hs}^2 + \norm{\nabla \zeta}_{\Hs}^2  \\
    & \leq C_R \left(\norm{\xi}_{\Hs}^2 + \norm{\rho}_{\Hs}^2 + \norm{\epsilono}_{\Hs}^2  + \norm{\zeta}_{\Hs}^2  + \norm{h_1}_{\Hs}^2 + \norm{h_2}_{\Hs}^2\right).
    \end{split}
\end{equation}
Testing \eqref{eq:lin_displacement} with $\epsilonot$ and employing the fact that $\A$ is positive definite, we have
    \begin{equation*}
        \begin{split}
            C_{\A} \norm{\epsilonot}_{\Hs}^2 &\leq \into \A \epsilonot : \epsilonot \dx\\
            &= \into \left(-B(\varphi,z)\epsilono:\epsilonot + \xi \cc_1 : \epsilonot + \zeta \cc_2 : \epsilonot \right) \dx.
        \end{split}
    \end{equation*}
Recalling that $\B$ is bounded and $\cc_1,\cc_2$ are uniformly bounded in $L^{\infty}(0,T;L^6)$ thanks to \Cref{remark:boundedness_coeff_lin}, we estimate the right-hand side with the H\"older and the Young inequalities. We get
    \begin{equation*}
        \begin{split}
             C_{\A} \norm{\epsilonot}_{\Hs}^2 
             &\leq \left(C \norm{\epsilono}_{\Hs} + \norm{\xi}_{L^3}\norm{\cc_1}_{L^6} + \norm{\zeta}_{L^3}\norm{\cc_2}_{L^6}\right) \norm{\epsilonot}_{\Hs}\\
             & \leq \delta \norm{\epsilonot}_{\Hs}^2 + C_{R,\delta} \left( \norm{\epsilono}_{\Hs}^2 + \norm{\xi}_{L^3}^2 + \norm{\zeta}_{L^3}^2\right),
        \end{split}
    \end{equation*}
    where $\delta$ is a small positive constant yet to be determined. By means of the interpolation inequality in \Cref{lemma:special_case_gagliardo_nirenberg} and again the Young inequality, we have
    \begin{equation}\label{eq:energy_ineq2_lin}
        \begin{split}
             C_{\A} \norm{\epsilonot}_{\Hs}^2 \leq  \delta \left(\norm{\epsilonot}_{\Hs}^2 + \norm{\nabla \xi}_{\Hs}^2 + \norm{\nabla \zeta}_{\Hs}^2\right)\\ + C_{R,\delta} \left( \norm{\epsilono}_{\Hs}^2 + \norm{\xi}_{\Hs}^2 + \norm{\zeta}_{\Hs}^2\right). 
        \end{split}
    \end{equation}
    Recalling that
    \begin{equation*}
        \norm{\epsilono}_{\Hs}^2 \leq \intt \norm{\epsilonot}_{\Hs}^2 \ds 
    \end{equation*}
    because $\oomega(0)=0$,
    summing equations \eqref{eq:enqergy_ineq1_lin} and \eqref{eq:energy_ineq2_lin}, and fixing $\delta$ small enough leads to 
    \begin{equation*}
        \begin{split}
             \frac{\text{d}}{\text{d}t}& \left( \norm{\xi}_{\Hs}^2 + \norm{ \rho}_{\Hs}^2 + \norm{\zeta}_{\Hs}^2\right) + \norm{\nabla \xi}_{\Hs}^2 + \norm{\nabla \rho}_{\Hs}^2 +  \norm{\epsilonot}_{\Hs}^2 + \norm{\nabla \zeta}_{\Hs}^2  \\
    & \leq C_R \left(\norm{\xi}_{\Hs}^2 + \norm{\rho}_{\Hs}^2  + \norm{\zeta}_{\Hs}^2   + \intt \norm{\epsilonot}_{\Hs}^2 \ds  + \norm{\hh}_{L^2(Q)}^2\right).
        \end{split}
    \end{equation*}
    Applying Gronwall inequality and then standard parabolic regularity estimates, we recover the a priori estimates we need to pass in the Galerkin discretization. This way, existence is proved as well as the estimate \eqref{eq:lin_sys_estimate} in the statement. Uniqueness follows from the estimate we have just proved and the fact that the system is linear.
\end{proof}

\begin{remark}
    For $\cchi \in \Ur$ fixed, it is convenient to denote the solution to the linearized state system associated with a perturbation $\hh \in \U$ as $(\xih,\rhoh,\oomegah,\zetah)$. From this result, it follows that the map 
    \begin{equation*}
        \U \subseteq L^2(Q) \times L^2(Q) \to \V, \qquad \hh \mapsto (\xih,\rhoh,\oomegah,\zetah),
    \end{equation*}
    is linear and continuous. 
\end{remark}

\section{Differentiability of the control-to-state operator}\label{section:diff_solution_op}
In this section, we will prove the Fréchet differentiability of the control-to-state operator. 

\begin{thm}\label{thm:diff_solution_op}
    The control-to-state operator $\S : \U \to \V$ is Fréchet differentiable in $\Ur$ and the Fréchet derivative of $\S$ in $\cchi \in \Ur$ is given by 
    \begin{equation*}
        D\S(\cchi)\hh = (\xih,\rhoh,\oomegah,\zetah)
    \end{equation*}
    for every $\hh \in \U$.
\end{thm}
\begin{proof}
    We consider a fixed and arbitrary $\cchi \in \Uad$ with $\S(\cchi)=(\varphi,\sigma,\uu,z)$. Our goal is to prove that
    \begin{equation}\label{eq:limit_diff}
        \lim_{\norm{\hh}_{\U} \to 0} \frac{\norm{\S(\cchi + \hh) - \S(\cchi) - (\xih,\rhoh,\oomegah,\zetah)}_{\V}}{\norm{\hh}_{\U}} = 0,
    \end{equation}
    from which the thesis follows.
    We introduce the notation $\S(\cchi + \hh) = (\varphih,\sigmah,\uuh,\zh)$ and 
    \begin{equation*}
        \Phih \coloneqq \varphih - \varphi - \xih, \quad 
        \lambdah \coloneqq \sigmah - \sigma - \rhoh, \quad
        \wwh \coloneqq \uuh - \uu - \oomegah, \quad \muh \coloneqq \zh - z - \zetah.
    \end{equation*}
    Since $\cchi$ belongs to $\Uad$ which is in turn contained in the open set $\Ur$, there exists a constant $C_{\cchi}$ such that, for every $\hh \in \Ur$ with $\norm{\hh}_{\U} \leq C_{\cchi}$, the control $\cchi + \hh$ still belongs to $\Ur$. Without loss of generality, since our aim is to pass to the limit as $\norm{\hh}_{\U}$ goes to $0$, we will consider only $\hh$ with a small norm in this sense. We are going to prove that
    \begin{equation}\label{eq:diff_thesis}
        \norm{(\Phih,\lambdah, \wwh, \muh)}_{\V} \leq C_R \norm{\hh}_{\U}^\frac{5}{4},
    \end{equation}
    which yields to the limit \eqref{eq:limit_diff}. To do so, we consider the PDE system satisfied by $(\Phih,\lambdah,\wwh,\muh)$ which can be trivially obtained by \Cref{thm:wellposedness} and \Cref{prop:existence_lin}. Explicitly, the following equations are satisfied
    \begin{subequations}\label{eq:system_diff}
        \begin{align}
        & \into \partial_t\Phih \eta \dx  + \into \nabla \Phih \cdot \nabla \eta \dx = \into [ A_1 + A_2] \eta \dx, \label{eq:phi_diff}\\
        &  \duality{\partial_t\lambdah}{\eta }_{\Vs} + \into \nabla \lambdah \cdot \nabla \eta \dx +  \intg \lambdah \eta \dA= \into [B_1 + B_2 + B_3] \eta \dx,\label{eq:sigma_diff}\\
        &  \into  \left[\A \epsilonwth + \C_1\right] : \varepsilon(\ttheta) \dx = 0,\label{eq:u_diff}\\
        & \into \partial_t \muh \eta \dx + \nabla \muh \cdot \nabla \eta \dx = \into \left[ D_1 + D_2\right] \eta \dx,\label{eq:z_diff}
    \end{align}
\end{subequations}
for every $\eta \in \Vs$ and $\ttheta \in \Vs_0$ as well as the initial conditions
\begin{equation}
    \Phih(0)=0, \quad \lambdah(0)=0,\quad \wwh(0)=0, \quad \muh(0)=0.
\end{equation}
Here we have introduced the notation:
\allowdisplaybreaks
\begin{align}
    A_1 =& U(\varphih,\sigmah,\zh,\chi_1) -  U(\varphi,\sigma,z,\chi_1)\nonumber\\
    &- \left[ U_{,\varphi}(\varphi,\sigma,z,\chi_1)\xih + U_{,\sigma}(\varphi,\sigma,z,\chi_1)\rhoh + U_{,z}(\varphi,\sigma,z,\chi_1)\zetah\right], \nonumber\\
    A_2 =& -\bigg[\varphih \Big(1-\frac{\varphih}{N}\Big) - \varphi \left(1-\frac{\varphi}{N}\right)\bigg]h_1,\label{eq:A2}\\
    B_1 = & \, K(\varphih,\sigmah,\zh) - K(\varphi,\sigma,z) - \left[K_{,\varphi}(\varphi,\sigma,z) \xih + K_{,\sigma}(\varphi,\sigma,z) \rhoh + K_{,z}(\varphi,\sigma,z) \zetah \right],\nonumber\\
    B_2 = & \left[ S(\varphih,\zh) - S(\varphi,z) - \left(S_{,\varphi}(\varphi,z) \xih + S_{,z}(\varphi,z) \zetah \right)\right] \chi_2,\nonumber\\
    B_3 = & [S(\varphih,\zh)-S(\varphi,z)]h_2,\label{eq:B3}\\
    \C_1 =& \B(\varphih,\zh)\epsilonuh - \B(\varphi,z)\epsilonu\nonumber\\
    &- \left[ \B_{,\varphi}(\varphi,z)\epsilonu\xih + \B(\varphi,z)\epsilonoh +\B_{,z}(\varphi,z)\epsilonu \zetah \right],\nonumber\\
    D_1 =& - \left[\beta(\zh) + \pi(\zh) - (\beta(z) + \pi(z)) -(\beta'(z) + \pi'(z))\zetah\right],\nonumber\\
    D_2 =& -\left[ \Psi(\varphih,\epsilonuh) - \Psi(\varphi,\epsilonu) - (\Psi_{,\varphi}(\varphi,\epsilonu)\xih + \Psi_{,\varepsilon}(\varphi,\epsilonu):\epsilonoh)\right]\nonumber.
\end{align}
The next step is testing each equation in \eqref{eq:system_diff} with a proper term and doing some estimates. For this reason, it is convenient to rewrite some of the known coefficient functions we have just introduced. To this end, we recall that, according to Taylor's theorem with an integral reminder, for a function $l \in W^{2,2}([0,1])$ it holds
\begin{equation*}
    l(1) = l(0) + l'(0) + \int_0^1 l''(s)(1-s) \ds.
\end{equation*}
Let's take $A_1$ into account. We introduce $\yyh=(\varphih,\sigmah,\zh,\chi_1)$, $ \yy=(\varphi,\sigma,z,\chi_1)$, and the function
\begin{equation*}
    l(s)=U(s\yyh + (1-s)\yy ),
\end{equation*}
which is $W^{2,\infty}$ because $U$ has this regularity and $s\yyh +(1-s)\yy$ is $L^{\infty}$.
If we apply the formula above, we get
\begin{equation*}
    \begin{split}
    U(\yyh) &= U(\yy) + \nabla U(\yy) \cdot (\yyh-\yy)\\
    &\quad+ \int_0^1 \left[D^2 U (s\yyh +(1-s)\yy) (\yyh -\yy)  \cdot (\yyh -\yy)\right](1-s) \ds\\
    & \eqcolon U(\yy) + \nabla U(\yy) \cdot (\yyh-\yy)+ \mathfrak{A}_1 (\yyh -\yy)  \cdot (\yyh -\yy).
    \end{split}
\end{equation*}
Notice that the matrix
\begin{equation*}
    \mathfrak{A}_1 = \int_0^1 D^2 U (s\yyh +(1-s)\yy) (1-s) \ds
\end{equation*}
as well as $\nabla U$ are bounded and their $L^{\infty}$-norm are uniformly controlled by a constant that depends on $R$. Comparing the equality we have just obtained with $A_1$ leads to 
\begin{equation}\label{eq:A1}
    \begin{split}
        A_1 &= U_{,\varphi}(\varphi,\sigma,z,\chi_1)\Phih + U_{,\sigma}(\varphi,\sigma,z,\chi_1)\lambdah+ U_{,z}(\varphi,\sigma,z,\chi_1)\muh\\
        &\quad + \mathfrak{A}_1(\varphih-\varphi,\sigmah-\sigma,\zh-z,0) \cdot (\varphih-\varphi,\sigmah-\sigma,\zh-z,0).
    \end{split}
\end{equation}
Proceeding in the same way, we have:
\begin{align}
    \begin{split}
        B_1 = & K_{,\varphi}(\varphi,\sigma,z) \Phih + K_{,\sigma}(\varphi,\sigma,z) \lambdah + K_{,z}(\varphi,\sigma,z) \muh \\
        &+ \mathfrak{B}_1 (\varphih-\varphi,\sigmah-\sigma,\zh-z) \cdot (\varphih-\varphi,\sigmah-\sigma,\zh-z), 
    \end{split}\label{eq:B1}\\
    B_2 = & \left[ S_{,\varphi}(\varphi,z) \Phih + S_{,z}(\varphi,z) \muh + \mathfrak{B}_2 (\varphih-\varphi,\zh-z) \cdot (\varphih-\varphi,\zh-z) \right] \chi_2,\label{eq:B2}\\
    \begin{split}
        \C_1 =&\B_{,\varphi}(\varphi,z)\epsilonu\Phih +  \B(\varphi,z)\epsilonwh +  \B_{,z}(\varphi,z)\epsilonu \muh\\
        &+\overline{\C}_1 (\varphih-\varphi, \epsilonuh-\epsilonu,\zh-z) \cdot(\varphih-\varphi, \epsilonuh-\epsilonu,\zh-z),
    \end{split}\label{eq:C2}\\
    D_1 =& - \left[(\beta'(z) + \pi'(z))\muh + \DD_1 (\zh-z)^2\right],\label{eq:D1}\\
    \begin{split}
        D_2 =& -\Big[\Psi_{,\varphi}(\varphi,\epsilonu)\Phih + \Psi_{,\varepsilon}(\varphi,\epsilonu):\epsilonwh\\
    &+ \DD_2 (\varphih-\varphi, \epsilonuh-\epsilonu) \cdot (\varphih-\varphi, \epsilonuh-\epsilonu)\Big].
    \end{split}\label{eq:D2}
\end{align}
Notice that the terms written in Fraktur font are the ones related to the integral of the hessian of the auxiliary function $l$, and, therefore, their dimensions change from case to case: for example, $\B_1$ is a matrix in $\RR^{3 \times 3}$, $\B_2$ a matrix in $\RR^{2\times 2}$, and $\DD_1$ is a scalar. Moreover, it is easy to check that all these terms are uniformly bounded in $L^{\infty}$ by a constant that depends on $R$ with the only exception of $\overline{\C}_1$, which is uniformly bounded in $L^{\infty}(L^6)$ because, even if $\B \in W^{2,\infty}$, the terms $\epsilonuh$, $\epsilonu$ are uniformly bounded only in this weaker norm. We will examine $\overline{\C}_1$ more closely later, addressing the estimate of $\C_1$.
We test equation \eqref{eq:phi_diff} with $\Phih$, obtaining
\begin{equation*}
    \frac{1}{2}\frac{\text{d}}{\text{d}t} \norm{\Phih}_{\Hs}^2 + \norm{\nabla \Phih}_{\Hs}^2 = \into (A_1 + A_2) \Phih \dx.
\end{equation*}
From equation \eqref{eq:A1}, 
\begin{equation*}
    |A_1| \leq C_R \left( |\Phih| + |\lambdah| + |\muh| + |\varphih-\varphi|^2 + |\sigmah-\sigma|^2 + |\zh-z|^2\right),
\end{equation*}
and from equation \eqref{eq:A2},
\begin{equation*}
    |A_2| \leq C_R |\varphih-\varphi|\,|h_1|.
\end{equation*}
Putting these elements together and using the H\"older inequality, we obtain
\begin{equation*}
    \begin{split}
        \frac{1}{2}\frac{\text{d}}{\text{d}t}&\norm{\Phih}_{\Hs}^2 + \norm{\nabla \Phih}_{\Hs}^2\\
        \leq& C_{R} \Big( \norm{\Phih}_{\Hs}+ \norm{\lambdah}_{\Hs} + \norm{\muh}_{\Hs}\Big)\norm{\Phih}_{\Hs}\\
        &+ C_{R} \Big(\norm{\varphih-\varphi}_{\Hs}\norm{\varphih-\varphi}_{L^6}+ \norm{\sigmah-\sigma}_{\Hs}\norm{\sigmah-\sigma}_{L^6} \\
         &+ \norm{\zh-z}_{\Hs}\norm{\zh-z}_{L^6} + \norm{\varphih-\varphi}_{L^6} \norm{h_1}_{\Hs}\Big)\norm{\Phih}_{L^3}.
    \end{split}
\end{equation*}
We estimate the term $\norm{\Phi}_{L^3}$ by means of the Gagliardo--Nirenberg inequality (see \Cref{lemma:special_case_gagliardo_nirenberg}), and then we apply the H\"older inequality. We have
\begin{equation}\label{eq:phi_diff_estimate}
     \begin{split}
        \frac{1}{2}&\frac{\text{d}}{\text{d}t} \norm{\Phih}_{\Hs}^2 + \norm{\nabla \Phih}_{\Hs}^2
        \leq C_R \Big(\norm{\lambdah}_{\Hs}^2 + \norm{\muh}_{\Hs}^2\\
        &+\norm{\varphih-\varphi}_{\Hs}^2\norm{\varphih-\varphi}_{L^6}^2+ \norm{\sigmah-\sigma}_{\Hs}^2\norm{\sigmah-\sigma}_{L^6}^2 + \norm{\zh-z}_{\Hs}^2\norm{\zh-z}_{L^6}^2 \Big)\\
        &+ C_R \norm{\varphih-\varphi}_{\Hs}^2\norm{h_1}_{L^6}^2+ C_{R,\delta}\norm{\Phih}_{\Hs}^2 + \delta \norm{\nabla \Phih}_{\Hs}^2
    \end{split}
\end{equation}
for a small parameter $\delta>0$. 
Testing \eqref{eq:sigma_diff} with $\lambdah$ leads to
\begin{equation*}
    \frac{1}{2}\frac{\text{d}}{\text{d}t}\norm{\lambdah}_{\Hs}^2 + \norm{\nabla \lambdah}_{\Hs}^2 \leq \frac{1}{2}\frac{\text{d}}{\text{d}t}\norm{\lambdah}_{\Hs} + \norm{\nabla \lambdah}_{\Hs}^2 + \norm{\lambdah}_{L^2_{\Gamma}}^2 = \into (B_1 + B_2 + B_3) \lambdah \dx. 
\end{equation*}
Proceeding as before, from equations \eqref{eq:B1} and \eqref{eq:B2} we get
\begin{equation*}
    |B_1| + |B_2| \leq C_R \left(|\Phih| + |\lambdah| + |\muh| + |\varphih-\varphi|^2 + |\sigmah-\sigma|^2 + |\zh-z|^2\right), 
\end{equation*}
where we have also employed the fact that $\norm{\chi_2}_{L^{\infty}} \leq R$. From equation \eqref{eq:B3} we derive
\begin{equation*}
    |B_3| \leq C\left(|\varphih-\varphi|+ |\zh-z|\right) |h_2|.
\end{equation*}
Consequentially, through the H\"older and the Young inequalities, then the Gagliardo--Nirenberg inequality, and again the Young inequality with a small positive parameter $\delta$, we have 
\begin{equation}\label{eq:lambda_diff_estimate}
    \begin{split}
        \frac{1}{2}&\frac{\text{d}}{\text{d}t}\norm{\lambdah}_{\Hs}^2 + \norm{\nabla \lambdah}_{\Hs}^2 \leq C_R \Big( \norm{\Phih}_{\Hs}^2 + \norm{\muh}_{\Hs}^2\\
        &+\norm{\varphih-\varphi}_{\Hs}^2\norm{\varphih-\varphi}_{L^6}^2+ \norm{\sigmah-\sigma}_{\Hs}^2\norm{\sigmah-\sigma}_{L^6}^2 + \norm{\zh-z}_{\Hs}^2\norm{\zh-z}_{L^6}^2 \Big)\\
        &+ C_R \left(\norm{\varphih-\varphi}_{L^6}^2 
        + \norm{\zh-z}_{L^6}^2 \right)\norm{h_2}_{\Hs}^2+ C_{R,\delta}\norm{\lambdah}_{\Hs}^2 + \delta \norm{\nabla \lambdah}_{\Hs}^2.
    \end{split}
\end{equation}
We test equation \eqref{eq:u_diff} with $\partial_t\wwh$, obtaining:
\begin{equation*}
    \begin{split}
        C_{\A} \norm{\epsilonwth}_{\Hs}^2 &\leq \into \A\epsilonwth : \epsilonwth \dx= -\into \C_1 :\epsilonwth \dx.
    \end{split}
\end{equation*}
Our goal is to perform suitable estimates of the right-hand side of this inequality. 
By equation \eqref{eq:C2}, 
\begin{equation*}
    \begin{split}
        |\C_1| \leq & C |\epsilonu|\left( |\Phih| + |\muh|\right) +  C|\epsilonwh|\\
        &+ |\overline{\C}_1 (\varphih-\varphi,\epsilonuh-\epsilonu,\zh-z) \cdot (\varphih-\varphi,\epsilonuh-\epsilonu,\zh-z)|.
    \end{split}
\end{equation*}
Let us analyze the last term on the right-hand side. We define $\yyh = (\varphih,\epsilonuh, \zh)$, $\yy = (\varphi,\epsilonu, z)$, the function 
\begin{equation*}
    L(\yy) = \B(\varphi,z)\epsilonu,
\end{equation*}
and the associated 
\begin{equation*}
    l(s)=L\big(s\yyh +(1-s)\yy\big)=\B\big(s\varphih+(1-s)\varphi,s\,\zh +(1-s)\big)\left[s\,\epsilonuh + (1-s)\epsilonu\right].
\end{equation*}
As done before, 
\begin{equation*}
    \overline{\C}_1 = \int_0^1 D^2L\big(s\yyh +(1-s)\yy\big)(1-s) \ds.
\end{equation*}
 Notice that $L$ is linear in $\epsilonu$, so the related second derivative vanishes. Thus, the quadratic term in $\epsilonuh-\epsilonu$ will not appear in $\C_1$.  Explicitly,
\begin{equation*}
   \begin{split}
        |\overline{\C}_1 (\yyh-\yy)\cdot(\yyh-\yy) | \leq &C \left(|\epsilonuh|+|\epsilonu|\right)\left[|\varphih-\varphi|^2 + |\zh-z|^2\right]\\
        &+C|\epsilonuh-\epsilonu|\left[|\varphih-\varphi| + |\zh-z|\right],
   \end{split}
\end{equation*}
because $\B$ belongs to $W^{2,\infty}$. Employing this inequality, we have
\begin{equation*}
    \begin{split}
         C_{\A} & \norm{\epsilonwth}_{\Hs}^2 \leq C \into \left( |\epsilonu|\left( |\Phih|  + |\muh|\right)|\epsilonwth|  + |\epsilonwh||\epsilonwth| \right)\dx\\
         &+C \into \left(|\epsilonuh-\epsilonu|\right)\left(|\varphih-\varphi| + |\zh - z|\right)|\epsilonwth|\dx\\
          &+C \into \left(|\epsilonuh|+|\epsilonu|\right)\left(|\varphih-\varphi|^2 + |\zh - z|^2 \right)|\epsilonwth|\dx\\
         \leq & C\Big[\norm{\epsilonu}_{L^6}\left(\norm{\Phih}_{L^3}+\norm{\muh}_{L^3}\right) + \norm{\epsilonwh}_{\Hs}\\
         &+ \norm{\epsilonuh-\epsilonu}_{L^4}\left(\norm{\varphih-\varphi}_{L^4} + \norm{\zh-z}_{L^4}\right)\\
          &+ \left( \norm{\epsilonuh}_{L^6} + \norm{\epsilonu}_{L^6} \right) \left(\norm{\varphih-\varphi}_{L^6}^2 + \norm{\zh-z}_{L^6}^2 \right) \Big] \norm{\epsilonwth}_{\Hs}.
    \end{split}
\end{equation*}
We recall that by \ref{thm:wellposedness_energy_estimate} the terms $\norm{\epsilonu}_{L^6}$, $\norm{\epsilonuh}_{L^6}$ are bounded by a constant that depends on $R$. By the Young inequality and the Gagliardo--Nirenberg interpolation inequality from \Cref{lemma:special_case_gagliardo_nirenberg}, we deduce
\begin{equation}\label{eq:w_diff_estimate}
    \begin{split}
         C_{\A}& \norm{\epsilonwth}_{\Hs}^2\\
         &\leq \delta \left(\norm{\nabla \Phih}_{\Hs}^2 + \norm{\nabla \muh}_{\Hs}^2 + \norm{\epsilonwth}_{\Hs}^2 \right)+ C_{R,\delta} \Big[\norm{\Phih}_{\Hs}^2+\norm{\muh}_{\Hs}^2\\
         &\quad+ \norm{\epsilonwh}_{\Hs}^2+ \norm{\epsilonuh-\epsilonu}_{L^4}^2(\norm{\varphih-\varphi}_{L^4}^2 + \norm{\zh-z}_{L^4}^2)\\
         &\quad+ \norm{\varphih-\varphi}_{L^6}^4 + \norm{\zh-z}_{L^6}^4 \Big]
    \end{split}
\end{equation}
for a small $\delta>0$.
Testing equation \eqref{eq:z_diff} with $\muh$ leads to
\begin{equation*}
    \frac{1}{2}\frac{\text{d}}{\text{d}t} \norm{\muh}_{\Hs}^2 + \norm{\nabla \muh}_{\Hs}^2 = \into (D_1 + D_2) \muh \dx.
\end{equation*}
Regarding $D_1$, we recall that thanks to the separation property we proved for $z$, the term $\beta'(z)+\pi'(z)$ is bounded. We have
\begin{equation*}
    |D_1| \leq C \left( |\muh| + |\zh - z|^2\right).
\end{equation*}
Turning our attention to $D_2$, we get
\begin{equation*}
    |D_2| \leq C \left(|\Phih| + |\epsilonwh| + |\varphih-\varphi|^2 + |\epsilonuh-\epsilonu|^2\right).
\end{equation*}
Thus, with standard argumentation, we deduce
\begin{equation}\label{eq:mu_diff_estimate}
    \begin{split}
        \frac{1}{2}&\frac{\text{d}}{\text{d}t} \norm{\muh}_{\Hs}^2 + \norm{\nabla \muh}_{\Hs}^2\\
        &\leq  \delta \norm{\nabla \muh}_{\Hs}^2 + C_{\delta} \Big(\norm{\muh}_{\Hs}^2 + \norm{\Phih}_{\Hs}^2+ \norm{\epsilonwh}_{\Hs}^2
        + \norm{\zh-z}_{\Hs}^2\norm{\zh-z}_{L^6}^2\\
        &\quad+ \norm{\varphih-\varphi}_{\Hs}^2 \norm{\varphih-\varphi}_{L^6}^2 + \norm{\epsilonuh-\epsilonu}_{\Hs}^2 \norm{\epsilonuh-\epsilonu}_{L^4}^2\Big).
    \end{split}
\end{equation}
We sum inequalities \eqref{eq:phi_diff_estimate}--\eqref{eq:mu_diff_estimate}, integrate in time over $(0,t)$ remembering that the initial values of the unknowns are zero, and move the terms multiplied by $\delta$ to the left-hand side, fixing a parameter small enough. This way we get the following: 
\begin{equation*}
    \begin{split}
        &\norm{\Phih}_{\Hs}^2 + \norm{\lambdah}_{\Hs}^2 + \norm{\muh}_{\Hs}^2 + \intt \left(\norm{\nabla \Phih}_{\Hs}^2 + \norm{\nabla \lambdah}_{\Hs}^2 + \norm{\epsilonwth}_{\Hs}^2 + \norm{\nabla \muh}_{\Hs}^2  \right) \ds\\
        &\leq C_R \intt\Big[\norm{\Phih}_{\Hs}^2 
        + \norm{\lambdah}_{\Hs}^2 + \norm{\epsilonwh}_{\Hs}^2 + \norm{\muh}_{\Hs}^2\\
        &\quad+\norm{\varphih-\varphi}_{\Hs}^2\norm{\varphih-\varphi}_{L^6}^2+ \norm{\sigmah-\sigma}_{\Hs}^2\norm{\sigmah-\sigma}_{L^6}^2 + \norm{\zh-z}_{\Hs}^2\norm{\zh-z}_{L^6}^2 \\
        &\quad+ \norm{\varphih-\varphi}_{\Hs}^2\norm{h_1}_{L^6}^2 + \left(\norm{\varphih-\varphi}_{L^6}^2 
        + \norm{\zh-z}_{L^6}^2 \right)\norm{h_2}_{\Hs}^2\\
        &\quad+ \norm{\epsilonuh-\epsilonu}_{L^4}^2(\norm{\varphih-\varphi}_{L^4}^2 + \norm{\epsilonuh-\epsilonu}_{\Hs}^2 + \norm{\zh-z}_{L^4}^2)\\
        &\quad+ \norm{\varphih-\varphi}_{L^6}^4 + \norm{\zh-z}_{L^6}^4 \Big] \ds.
    \end{split}
\end{equation*}
We recall that
\begin{equation*}
    \norm{\epsilonwh}_{\Hs}^2 \leq C \ints \norm{\epsilonwth}_{\Hs}^2 \dtau,
\end{equation*}
and that, thanks to \ref{thm:wellposedness_continuous_dependence} in \Cref{thm:wellposedness}, it holds
\begin{multline*}
    \norm{\varphih-\varphi}_{L^{\infty}(\Hs)} +  \norm{\sigmah - \sigma}_{L^{\infty}(\Hs)} +\norm{\epsilonuh-\epsilonu}_{L^{\infty}(\Hs)}\\
    + \norm{\zh-z}_{L^{\infty}(\Hs)}  \leq C_R \norm{\hh}_{L^2(Q)},
\end{multline*}
and that, by \eqref{eq:add_continuous_estimate_inf4} in \Cref{remark:add_continuous_estimate}, we have
\begin{equation*}
    \norm{\varphih-\varphi}_{L^{\infty}(L^4)} + \norm{\epsilonuh-\epsilonu}_{L^{\infty}(L^4)} +  \norm{\zh-z}_{L^{\infty}(L^4)}\leq C_R \norm{\hh}_{L^2(Q)}^{\frac{1}{4}}.
\end{equation*}
Finally, we observe that
\begin{equation*}
    \begin{split}
        \intt &\norm{\varphih-\varphi}_{L^6}^4 \ds \leq \intt \norm{\varphih-\varphi}_{L^{\infty}(\Omega)}^2 \norm{\varphih-\varphi}_{L^3}^2 \ds\\
        & \leq \norm{\varphih-\varphi}_{L^4(L^{\infty})}^2 \norm{\varphih-\varphi}_{L^4(L^3)}^2 \leq C_R \norm{\hh}_{L^2(Q)}^{\frac{1}{2}} \norm{\hh}_{L^2(Q)}^2 = C_R \norm{\hh}_{L^2(Q)}^{\frac{5}{2}}
    \end{split}
\end{equation*}
where we have applied the H\"older inequality and, in the last passage, we have combined \eqref{eq:add_continuous_estimate_4inf} and \eqref{eq:add_continuous_estimate_43} from \Cref{remark:add_continuous_estimate}. The same inequality holds for $\zh-z$. Thus, we obtain
\begin{equation*}
    \begin{split}
        &\norm{\Phih}_{\Hs}^2 + \norm{\lambdah}_{\Hs}^2 + \norm{\muh}_{\Hs}^2 + \intt \left(\norm{\nabla \Phih}_{\Hs}^2 + \norm{\nabla \lambdah}_{\Hs}^2 + \norm{\epsilonwth}_{\Hs}^2 + \norm{\nabla \muh}_{\Hs}^2  \right) \ds\\
        &\leq C_R \bigg\{ \intt\Big[\norm{\Phih}_{\Hs}^2 
        + \norm{\lambdah}_{\Hs}^2 + \norm{\muh}_{\Hs}^2\Big] \ds + \intt \ints \norm{\epsilonwth}_{\Hs}^2 \dtau \ds \\
        &\quad+  \norm{\hh}_{L^2(Q)}^2 \intt \left[\norm{\varphih-\varphi}_{L^6}^2+ \norm{\sigmah-\sigma}_{L^6}^2 + \norm{\zh-z}_{L^6}^2 \right] \ds \\
        &\quad+ \norm{\hh}_{L^2(Q)}^2\norm{h_1}_{L^{2}(L^6)}^2 + \norm{h_2}_{L^{\infty}(\Hs)}^2 \intt \left[\norm{\varphih-\varphi}_{L^6}^2 +
        \norm{\zh-z}_{L^6}^2 \right] \ds \\
        &\quad +\norm{\hh}_{L^2(Q)}^{\frac{1}{2}} \intt \left[\norm{\varphih-\varphi}_{L^4}^2+ \norm{\zh-z}_{L^4}^2 \right] \ds  + \norm{\hh}_{L^2(Q)}^4 + \norm{\hh}_{L^2(Q)}^{\frac{5}{2}}\bigg\}.
    \end{split}
\end{equation*}
Again, we recall that from \ref{thm:wellposedness_continuous_dependence} in \Cref{thm:wellposedness} we know
\begin{equation*}
    \begin{split}
        \norm{\varphih-\varphi}_{L^{2}(\Vs)} +  \norm{\sigmah - \sigma}_{L^{2}(\Vs)} + \norm{\zh-z}_{L^{2}(\Vs)}  \leq C_R \norm{\hh}_{L^2(Q)},
    \end{split}
\end{equation*}
and that $\Vs \hookrightarrow L^4, L^6$. Finally, we obtain:
\begin{equation*}
    \begin{split}
        \norm{\Phih}_{\Hs}^2 &+ \norm{\lambdah}_{\Hs}^2 + \norm{\muh}_{\Hs}^2 + \intt \left(\norm{\nabla \Phih}_{\Hs}^2 + \norm{\nabla \lambdah}_{\Hs}^2 + \norm{\epsilonwth}_{\Hs}^2 + \norm{\nabla \muh}_{\Hs}^2  \right) \ds\\
        &\leq C_R \bigg\{ \intt\Big[\norm{\Phih}_{\Hs}^2 
        + \norm{\lambdah}_{\Hs}^2 + \norm{\muh}_{\Hs}^2\Big] \ds + \intt \ints \norm{\epsilonwth}_{\Hs}^2 \dtau \ds \\
        &\quad+  \norm{\hh}_{L^2(Q)}^4 + \norm{\hh}_{L^2(Q)}^2 \left(\norm{h_1}_{L^2(\Vs)}^2 +  \norm{h_2}_{L^{\infty}(\Hs)}^2\right) + \norm{\hh}_{L^2(Q)}^{\frac{5}{2}}\bigg\}.
    \end{split}
\end{equation*}
By means of the Gronwall inequality, we have:
\begin{equation*}
        \begin{split}
            \norm{\Phih}_{\Hs}^2 &+ \norm{\lambdah}_{\Hs}^2 + \norm{\muh}_{\Hs}^2 + \intt \left(\norm{\nabla \Phih}_{\Hs}^2 + \norm{\nabla \lambdah}_{\Hs}^2 + \norm{\epsilonwth}_{\Hs}^2 + \norm{\nabla \muh}_{\Hs}^2  \right) \ds\\
            & \leq C_R \left( \norm{\hh}_{L^2(Q)}^4 + \norm{\hh}_{L^2(Q)}^2 \left(\norm{h_1}_{L^2(\Vs)}^2 +  \norm{h_2}_{L^{\infty}(\Hs)}^2\right) + \norm{\hh}_{L^2(Q)}^{\frac{5}{2}}\right) ,
        \end{split}
\end{equation*}
from which \eqref{eq:diff_thesis} follows. Therefore, the proof of \Cref{thm:diff_solution_op} is complete. 
\end{proof}

\noindent From this result, it follows that the reduced cost functional $\J$ is Fréchet differentiable over the set $\Ur$. Since $\Uad$ is a closed and convex subset of $\U$, we can prove the following result. 

\begin{cor}\label{cor:necessary_condition_partial}
    Let $\cchi^* \in \Uad$ be an optimal control for the control problem with the associated state $\S(\cchi^*) = (\varphi^*,\sigma^*,\uu^*,z^*)$. Then, the following inequality is satisfied 
    \begin{equation}\label{eq:necessary_condition_partial}
        \begin{split}
            \alpha_1& \intTo (\varphi^* -\varphi_Q) \xi \dxt + \alpha_2 \into (\varphi^*(T) -\varphi_{\Omega}) \xi(T) \dx + \alpha_3 \into \xi(T) \dx\\
            &+ \alpha_4 \intTo (\sigma^* -\sigma_Q) \rho \dxt + \alpha_5 \into (\sigma^*(T) -\sigma_{\Omega}) \rho(T) \dx\\
            &+ \alpha_6 \intTo \left[ \frac{1}{2} \gamma'(\varphi^*) \varepsilon(\uu^*): \varepsilon(\uu^*) \xi + \gamma(\varphi^*)\varepsilon(\uu^*) : \epsilono \right] \dxt
            \\
            &+ \alpha_7\intTo (z^* -z_Q) \zeta \dxt + \alpha_8 \into z^*(T) \dx\\
            &+ \alpha_9 \intTo \cchi \cdot (\cchi - \cchi^*) \dxt \geq 0 
        \end{split}
    \end{equation}
    for every $\cchi \in \Uad$,
    where $(\xi,\rho,\oomega,\zeta)$ is the unique solution to the linearized system in $\cchi^*$ for $\hh = \cchi - \cchi^*$.
\end{cor}
\begin{proof}
First of all, we recall that the cost functional $\mathcal{J}$ is well-defined over the space
\begin{multline*}
    \Big[C^0([0,T];\Hs) \times C^0([0,T];\Hs) \times L^2(0,T;\Vs) \times C^0([0,T];\Hs) \Big]\\ \times \Big[L^2(0,T;\Hs) \times L^2(0,T;\Hs)\Big],
\end{multline*}
where it is also Fréchet differentiable. 
Moreover, from \Cref{thm:diff_solution_op}, the solution operator $\S:\U \to \V$ is Fréchet differentiable in $\Ur$. Since, from standard embedding results, $\V$ is continuously embedded in 
\begin{equation*}
    C^0([0,T];\Hs) \times C^0([0,T];\Hs) \times L^2(0,T;\Vs) \times C^0([0,T];\Hs),
\end{equation*}
 $\S$ is as well Fréchet differentiable in $\Ur$ if it is seen as a functional between $\U$ and this larger space.  Thus, the reduced cost functional $\J$ is Fréchet differentiable in $\Ur$ and, by the chain rule,
\begin{equation*}
    \begin{split}
        D\J(\cchi^*)&[\cchi-\cchi^*] =  \alpha_1 \intTo (\varphi^* -\varphi_Q) \xi \dxt + \alpha_2 \into (\varphi^*(T) -\varphi_{\Omega}) \xi(T) \dx\\
            & + \alpha_3 \into \xi(T) \dx + \alpha_4 \intTo (\sigma^* -\sigma_Q) \rho \dxt + \alpha_5 \into (\sigma^*(T) -\sigma_{\Omega}) \rho(T) \dx\\
            &+ \alpha_6 \intTo \left[\frac{1}{2} \gamma'(\varphi^*) \varepsilon(\uu^*): \varepsilon(\uu^*) \xi + \gamma(\varphi^*)\varepsilon(\uu^*) : \epsilono \right] \dxt
            \\
            &+ \alpha_7\intTo (z^* -z_Q) \zeta \dxt + \alpha_8 \into \zeta(T) \dx + \alpha_9 \intTo \cchi \cdot (\cchi - \cchi^*) \dxt,
    \end{split}
\end{equation*}
where $\cchi^*$ a the optimal control, $\cchi$ is any admissible control, and $(\xi,\rho,\oomega,\zeta)$ is the solution to the system linearized in  $\cchi^*$ with $\hh = \cchi - \cchi^*$.
From the optimality of $\cchi^*$ and the convexity of $\Uad$, we obtain the trivial inequality
\begin{equation*}
        \J(\cchi^*+t(\cchi-\cchi^*)) - \J(\cchi^*) \geq 0
\end{equation*}
for every $t \in (0,1)$. Dividing by $t$ and passing to the limit as $t \to 0^+$ leads to
\begin{equation*}
   \begin{split}
        D\J(\cchi^*)[\cchi-\cchi^*] \geq 0.
   \end{split}
\end{equation*}
\end{proof}

\noindent The next step is simplifying the expression \eqref{eq:necessary_condition_partial}, removing the linearized variables. In fact, even if the inequality just derived represents a first-order necessary condition for optimality, it does not provide a practically efficient characterization. Specifically, for every element $\cchi$ of the admissible control space $\Uad$, it requires solving the corresponding linearized system for $\hh = \cchi - \cchi^*$. This approach is computationally demanding, particularly in high-dimensional control spaces, highlighting the need for a reformulation. To this end, we need to introduce the adjoint system.  

\section{The adjoint system and first-order necessary optimality conditions}\label{section:adj_system}

The adjoint system associated with an optimal control $\cchi^* \in \Uad$ and its corresponding solution to the state system $(\varphi^*,\sigma^*,\uu^*,z^*) = \S(\cchi^*)$ is given by
\begin{subequations}\label{eq:adjoint_problem}
    \begin{align}
        & \begin{aligned}
            &-\partial_t\q - \Delta \q =  a_1 \q + b_1 \r + d_1 \s + \cc_1 : \epsilonv \\
            &\phantom{-\partial_t\q - \Delta \q = } + \alpha_1(\varphi^*-\varphi_Q) + \frac{\alpha_6}{2} \gamma'(\varphi^*) \varepsilon(\uu^*): \varepsilon(\uu^*),
        \end{aligned} \label{eq:adj_phi}\\
        & -\partial_t\r - \Delta \r = a_2\q + b_2 \r + \alpha_4(\sigma^*-\sigma_Q),\label{eq:adj_lactate}\\
        &  -\diver\left[- \A \epsilonvt + \B(\varphi^*,z^*)\epsilonv  \right] =  -\diver \left[ \dd_2 \s + \alpha_6\gamma(\varphi^*)\varepsilon(\uu^*)\right] , \label{eq:adj_displacement}\\
        & -\partial_t \s - \Delta \s =  a_3 \q + b_3 \r + d_3  \s +     \cc_2 : \epsilonv + \alpha_7(z^*-z_Q), \label{eq:adj_damage}
    \end{align}
\end{subequations}
coupled with the boundary conditions
\begin{equation}\label{eq:adjoint_boundary_conditions}
    \partial_{\nnu} \q = 0, \quad \partial_{\nnu} \r = - \r, \quad \vv = 0, \quad \partial_{\nnu} \s = 0,
\end{equation}
and with the final conditions
\begin{equation}\label{eq:adjoint_final_conditions}
    \q(T) = \alpha_2 (\varphi^*(T) - \varphi_{\Omega}) + \alpha_3, \quad \r(T) = \alpha_5 (\sigma^*(T)-\sigma_{\Omega}), \quad \vv(T) = 0, \quad \s(T) = \alpha_8.
\end{equation}

\begin{prop} \label{prop:existence_adj}
     Let $\cchi^*$ be an optimal control with the associated solution to the state system $(\varphi^*,\sigma^*,\uu^*,z^*) = \S(\cchi^*)$. The adjoint system \eqref{eq:adjoint_problem}--\eqref{eq:adjoint_final_conditions} has a unique weak solution $(\q, \r, \vv, \s)$ in the sense that
    \begin{gather*}
            \q \in   H^1(0,T;\Hs) \cap L^{\infty}(0,T; \Vs) \cap L^{2}(0,T;\Ws),\\
            \r \in H^1(0,T;\Vs') \cap L^{\infty}(0,T;\Hs) \cap L^2(0,T;\Vs), \\
            \vv \in W^{1,\infty}(0,T;\Vs_0),\\
            \s \in   H^1(0,T;\Hs) \cap L^{\infty}(0,T; \Vs) \cap L^{2}(0,T;\Ws) 
    \end{gather*}
    with 
   \begin{gather*}
        \q(T) = \alpha_2 (\varphi^*(T) - \varphi_{\Omega}) + \alpha_3, \quad \r(T) = \alpha_5 (\sigma^*(T)-\sigma_{\Omega}), \quad \vv(T) = 0, \quad \s(T) = \alpha_8,
   \end{gather*}
   such that
    \begin{subequations}\label{eq:problem_in_eq_with_spaces}
        \begin{align}
        &\begin{aligned}
            & \into \left(-\partial_t\q \eta +  \nabla \q \cdot \nabla \eta \right)\dx = \into \left[ a_1 \q + b_1 \r + d_1 \s + \cc_1 : \epsilonv\right]\eta \dx\\
        & \phantom{\into \left(-\partial_t\q \eta +  \nabla \q \cdot \nabla \eta \right)\dx =}+\into \left[\alpha_1(\varphi^*-\varphi_Q) + \frac{\alpha_6}{2}\gamma'(\varphi^*)\varepsilon(\uu^*):\varepsilon(\uu^*)\right] \eta \dx, \\
        \end{aligned}\label{eq:phi_adj_with_spaces}\\
        &\begin{aligned}
            &-\duality{\partial_t\r}{\eta }_{\Vs} + \into \nabla \r \cdot \nabla \eta \dx +  \intg \r \, \eta \dA \\
          & \phantom{ -\duality{\partial_t\r}{\eta }_{\Vs} + \into \nabla \r \cdot \nabla \eta \dx + } = \into \left[a_2\q + b_2 \r + \alpha_4(\sigma^*-\sigma_Q)\right] \eta \dx,
        \end{aligned} \label{eq:sigma_adj_with_spaces}\\
        &  \into  \left[-\A \epsilonvt + \B(\varphi^*,z^*)\epsilonv  \right] : \varepsilon(\ttheta) \dx = \into  \left[ \dd_2 \s + \alpha_6\gamma(\varphi^*)\varepsilon(\uu^*)\right] : \varepsilon(\ttheta) \dx,\label{eq:u_adj_with_spaces}\\
        & \into \left(-\partial_t \s \eta + \nabla \s \cdot \nabla \eta \right)\dx = \into \left[  a_3 \q + b_3 \r + d_3  \s +     \cc_2 : \epsilonv + \alpha_7(z^*-z_Q)\right] \eta \dx,\label{eq:z_adj_with_spaces}
    \end{align}
\end{subequations}
a.e. in $(0,T)$, for every $\eta \in \Vs$ and $\ttheta \in \Vs_0$.
\end{prop}

\noindent The well-posedness of the adjoint system can be proved as in \Cref{prop:existence_lin}, thus we will omit the proof.

\begin{thm}[First-order necessary conditions of optimality]
    Let $\cchi^* \in \Uad$ be an optimal control for the control problem with the associated state $\S(\cchi^*) = (\varphi^*,\sigma^*,\uu^*,z^*)$. Then, the following inequality is satisfied 
     \begin{equation}\label{eq:necessary_condition_2}
         \begin{split}
             &-\intTo \varphi^*\Big(1-\frac{\varphi^*}{N}\Big) (\chi_1-\chi_1^*) \q \dxt\\
             &\quad+ \intTo S(\varphi^*,z^*) (\chi_2 -\chi_2^*) \r \dxt
             + \alpha_9 \intTo \cchi \cdot (\cchi - \cchi^*) \dxt \geq 0 
         \end{split}
    \end{equation}
    for every $\cchi \in \Uad$,
    where $\q$ and $\r$ are respectively the solution's first and third components to the adjoint system associated with $\cchi^*$.
\end{thm}
\begin{proof}
We test the equations of the linearized system with the solution of the adjoint system and subtract the equations of the adjoint system tested with the solution of the linearized system. Then, we sum up all the equations we have obtained. 
Explicitly, we test \eqref{eq:phi_lin_with_spaces} with $\q$,
\eqref{eq:sigma_lin_with_spaces} with $\r$, \eqref{eq:u_lin_with_spaces} with $\vv$ and \eqref{eq:z_lin_with_spaces} with $\s$. In the same way, we test \eqref{eq:phi_adj_with_spaces} with $\xi$, \eqref{eq:sigma_adj_with_spaces} with $\rho$, \eqref{eq:u_adj_with_spaces} with $\oomega$ and \eqref{eq:z_adj_with_spaces} with $\zeta$. Some terms cancel out and we have:
\begin{equation*}
    \begin{split}
        \frac{\text{d}}{\text{d}t}&\left[\into \left(\xi \q + \rho \r +\A\epsilono :\epsilonv +  \zeta \s\right) \dx \right]\\ 
        &= \into \left(a_4 h_1 \q + b_4 h_2 \r\right) \dx - \alpha_1 \into (\varphi^*-\varphi_{Q})\xi \dx - \alpha_4 \into (\sigma^*-\sigma_{Q})\rho \dx\\
        &\quad - \alpha_6 \into \left(\frac{1}{2} \gamma'(\varphi^*)\varepsilon(\uu^*):\varepsilon(\uu^*)\xi + \gamma(\varphi^*)\varepsilon(\uu^*):\epsilono\right) \dx - \alpha_7 \into (z^*-z_{Q})\zeta \dx.
    \end{split}
\end{equation*}
Integrating in time over the interval $(0,T)$, exploiting the initial and final  conditions, and moving some terms to the left-hand side, we finally have
\begin{equation*}
    \begin{split}
        \bigg[&\alpha_2 \into (\varphi^*(T)-\varphi_{\Omega})\xi(T) \dx + \alpha_3\into \xi(T) \dx + \alpha_5 \into (\sigma^*(T)-\sigma_{\Omega})\rho(T) \dx\\
        &\quad+ \alpha_8 \into \zeta(T) \dx \bigg] + \bigg[\alpha_1 \intTo (\varphi^*-\varphi_{Q})\xi \dxt + \alpha_4 \intTo (\sigma^*-\sigma_{Q})\rho \dxt\\
        &\quad+ \alpha_6 \intTo \left( \frac{1}{2} \gamma'(\varphi^*)\varepsilon(\uu^*):\varepsilon(\uu^*)\xi + \gamma(\varphi^*)\varepsilon(\uu^*):\epsilono \right)\dxt\\
        &\quad+ \alpha_7 \intTo (z^*-z_{Q})\zeta \dxt\bigg] = \intTo (a_4 h_1 \q + b_4 h_2 \r) \dxt.
    \end{split}
\end{equation*}
Combining this equality with the inequality stated in \Cref{cor:necessary_condition_partial}, we obtain the thesis.
\end{proof}

\section*{Acknowledgments}
G.C. wishes to acknowledge the partial financial support received from  International Research Laboratory LYSM, IRL 2019 CNRS/INdAM. G.C. is a member of GNAMPA (Gruppo Nazionale per l’Analisi Matematica, la Probabilità e le loro Applicazioni) of INdAM (Istituto Nazionale di Alta Matematica).

\printbibliography

\end{document}